\documentclass[a4paper,12pt, reqno]{amsart}
\usepackage{amsfonts}
\usepackage{amsmath, array, bigints}
\usepackage{amssymb}
\usepackage{mathrsfs, mathtools}
 \usepackage{graphicx}
\usepackage{enumerate}

\usepackage[usenames]{color}
\makeatletter
\newcommand{\setword}[2]{%
  \phantomsection
  #1\def\@currentlabel{\unexpanded{#1}}\label{#2}%
}
\makeatother
\newtheorem{thm}{Theorem}[section]
\newtheorem{cor}[thm]{Corollary}
\newtheorem{lem}[thm]{Lemma}
\newtheorem{prop}[thm]{Proposition}

\numberwithin{equation}{section}
\usepackage[english]{babel} 
\usepackage{blindtext}

\theoremstyle{definition}
\newtheorem{definition}[thm]{Definition}
\newtheorem{rem}[thm]{Remark}

\usepackage[colorlinks]{hyperref}

\begin{document}

\allowdisplaybreaks 
\def\dist{{\operatorname{dist}}}
\renewcommand{\d}{\:\! \mathrm{d}}


 \title[Infinitely many solutions for nonlinear superposition operators]{Infinitely many solutions for nonlinear superposition operators of mixed fractional order involving critical exponent}

 \author[Souvik Bhowmick, Sekhar Ghosh, and Vishvesh Kumar]{Souvik Bhowmick, Sekhar Ghosh and Vishvesh Kumar}

\address[Souvik Bhowmick]{Department of Mathematics, National Institute of Technology Calicut, Kozhikode, Kerala, India - 673601}
\email{souvikbhowmick2912@gmail.com / souvik\_p230197ma@nitc.ac.in}

\address[ Sekhar Ghosh]{Department of Mathematics, National Institute of Technology Calicut, Kozhikode, Kerala, India - 673601}
\email{sekharghosh1234@gmail.com / sekharghosh@nitc.ac.in}
\address[Vishvesh Kumar]{Department of Mathematics: Analysis, Logic and Discrete Mathematics, Ghent University, Ghent, Belgium \newline
and \newline
Department of Mathematical Sciences, Indian Institute of Technology (BHU), Varanasi, Uttar Pradesh, 221005, India}
\email{vishveshmishra@gmail.com / vishvesh.mat@iitbhu.ac.in}
\date{}

\begin{abstract}
This paper addresses a class of elliptic problems involving the superposition of nonlinear fractional operators with the critical Sobolev exponent in the sublinear regimes. We establish the existence of infinitely many nontrivial weak solutions using a variational framework combining a truncation argument with the notion of genus. A central part of our analysis is the verification of the Palais--Smale (PS) condition for the associated energy functional for every $q \in (1, p_{s_\sharp}^*)$, despite the challenges posed by the lack of compactness due to the critical exponent. The results obtained in the paper are new even in the classical case $p = 2$, highlighting the broader applicability of the methods developed here.

\end{abstract}


\maketitle

\tableofcontents
\section{Introduction and main results}

The existence of infinitely many nonnegative solutions to elliptic problems involving local, nonlocal, and combined local-nonlocal operators has been extensively studied under Dirichlet boundary conditions. The symmetric mountain pass theorem, a variant of Lusternik-Schnirelmann theory, was initially developed by Clark \cite{C1972} (see also Amann \cite{A1972}) and plays a pivotal role in obtaining infinitely many solutions to such PDEs. The study by Ambrosetti and Rabinowitz \cite{AR: 1973} reignited interest in exploring the existence of infinitely many solutions for nonlinearities satisfying the well-known Ambrosetti-Rabinowitz (AR) condition. In \cite{AR: 1973}, Ambrosetti and Rabinowitz considered the following problem:
\begin{equation} \label{AR prob}
\begin{cases}
&-\Delta u = f(x,u) \quad \text{in} \,\, \Omega,\\
&u=0 \quad \text{on} \,\, \partial \Omega,\\
&u>0 \quad \text{in} \,\,  \Omega,
\end{cases}
\end{equation}
where $f$ is superlinear near the origin and sublinear at infinity. The results are obtained by employing a symmetric mountain pass theorem in conjunction with the notion of genus (see Rabinowitz \cite{R1986}). Kajikiya \cite{K2005, K2006} extended the study to the nonlinear setup involving the $p$-Laplacian for $1<p<\infty$. For a fruitful application of the genus theory combined with the symmetric mountain pass theorem and the Lusternick-Schnirelmann theory for proving infinitely many solutions, we refer to Garc{\'i}a Azorero and Peral Alonso \cite{AP1991}, Ambrosetti, Brezis, and Cerami \cite{ABC1994}, and Bartsch and Willem \cite{BW1994}.

The nonlocal extension to the problem \eqref{AR prob} is studied by Binlin {\it et al.} \cite{ZBS2015} by considering the following problem:
\begin{equation} \label{Zhang prob}
\begin{cases}
&-\mathcal{L}_K u = f(x,u) \quad \text{in} \,\, \Omega,\\
&u=0 \quad \text{on} \,\, \mathbb{R}^N\setminus \Omega,
\end{cases}
\end{equation}
where $\mathcal{L}_K$ represents a class of nonlocal operators defined by $$\mathcal{L}_Ku(x)=\int_{\mathbb{R}^N}(u(x+y)+u(x-y)-2u(x))K(y)dy$$ and $K$ is a generalised fractional kernel.

We now turn our attention to elliptic problems having a critical Sobolev exponent, known as the Brezis-Nirenberg type problems. Following the celebrated work of Brezis and Nirenberg \cite{BN83}, there has been significant attention to elliptic problems involving critical exponents. In \cite{BN83}, the authors investigated the following problem:
\begin{equation} \label{EucBNL}
\begin{cases}
&-\Delta u = |u|^{2^*-2} u+\lambda u \quad \text{in} \,\, \Omega,\\&
u=0 \quad \text{on} \,\, \partial \Omega,
\end{cases}
\end{equation}
where $\lambda \in \mathbb{R}$, $\Omega\subset\mathbb{R}^N$ is a smooth bounded domain  with $N\geq 3$ and $2^* := \frac{2N}{N-2}$ is the critical Sobolev exponent. Brezis and Nirenberg proved that a small perturbation to the Yamabe problem guarantees the existence of a positive solution. Specifically, they proved that problem \eqref{EucBNL} has a positive solution for $\lambda \in (0, \lambda_1)$ when $N \geq 4$, where $\lambda_1$ represents the principle eigenvalue of $(-\Delta)$ in $H^1_0(\Omega)$. In fact, when $N=3$, they obtained a stronger conclusion saying that there exists $\lambda_* \in (0, \lambda_1)$ such that for every $\lambda \in (\lambda_*, \lambda_1)$, problem \eqref{EucBNL} has a positive solution.  Moreover, they proved that problem \eqref{EucBNL} possesses a positive solution if and only if $\lambda \in (\lambda_1/4, \lambda_1)$ when $\Omega$ is a ball. On the other hand, if $\lambda \notin (0, \lambda_1)$, it is evident from the Poho\v{z}aev identity that problem \eqref{EucBNL} has no positive solution. Garc{\'i}a Azorero and  Peral Alonso  \cite{AP1991, AP1987} extended the results to the quasilinear setup involving the $p$-Laplacian to the following problem: 
\begin{equation} \label{P1}
\begin{cases}
&-\Delta_p u = |u|^{p^*-2} u+\lambda |u|^{q-2}u \quad \text{in} \,\, \Omega,\\&
u=0 \quad \text{on} \,\, \partial \Omega,
\end{cases}
\end{equation}
where $\Omega\subset \mathbb{R}^N$ is a bounded domain, $1<p<N$, $0<\lambda$, $1<q<p^*:= \frac{pN}{N-p}$ and $\Delta_p u=\nabla \cdot (|\nabla u|^{p-2} \nabla u)$. In \cite{AP1991}, Garc{\'i}a Azorero and  Peral Alonso proved that for \(1<q<p\), there exists a $\lambda_1>0$ such that for any $\lambda\in(0,\lambda_1)$, there are infinitely many nontrivial solutions to the problem \eqref{P1} using the properties of genus and truncated energy functional. Furthermore, for $p<q<p^*$, there exists $\lambda_2>0$ such that the problem \eqref{P1} has a nontrivial solution for every $\lambda> \lambda_2$. When $p=2$, Ambrosetti, Brezis, and Cerami \cite{ABC1994} established the existence of a $\lambda^*>0$ such that for $\lambda \in (0, \lambda^*)$, there are infinitely many solutions when $q < p < p^*$. For $p=2$, Bartsch and Willem \cite{BW1994} developed a new critical theorem that established the existence of infinitely many solutions without assuming a critical exponent. Degiovanni and Lancelotti \cite{Degiovanni2007} proved the existence of a nontrivial solution when $N\geq p^2$ and $\lambda>\lambda_1$ is not an eigenvalue and when $N^2/(N+1)>p^2$ and $\lambda\geq\lambda_1$. Egnell \cite{Egnell1988}, examined the case $p<N<p^2$ and proved that there exists $\lambda_*>0$ such that the problem \eqref{P1} has a positive solution for all $\lambda\in(\lambda_*,\lambda_1)$ with $p=q$. The existence of infinitely many positive solutions is studied by Cao {\it et al.}  \cite{CaoPengYan2013} for $N>p^2+p$, whereas for $p<q<p^*$, infinitely many solutions have been obtained by Garc{\'i}a Azorero and  Peral Alonso  \cite{AP1987} for any $\lambda>0$. Alves and Ding \cite{AlvesDing2010}, employed the Ljusternik–Schnirelmann category to prove that there exists $\bar{\lambda}$ such that for all $\lambda\in(0,\bar{\lambda})$, the problem \eqref{P1} has at least $cat(\Omega)$ positive solutions when $N\geq p^2$ and $2\leq p\leq q<p^*$. Schechter and Zou \cite{SZ2010} (see also \cite{SM2014} for $p=2$) proved when $p=q=2$ and $N\geq 7$, the problem \eqref{P1} has infinitely many solutions. In particular, if $\lambda\geq \lambda_1$, it has only sign-changing solutions except zero. Recently, He {\it et al.} \cite{HHZ2020} extended the results to the nonlinear case and proved that if problem \eqref{P1} admits a positive solution, then it has infinitely many sign-changing solutions.

The study of positive solutions and sign-changing solutions to the following nonlocal Brezis-Nirenberg type problems has also been studied extensively, involving the fractional Laplacian and fractional $p$-Laplacian:
\begin{equation} \label{P2}
\begin{cases}
&(-\Delta_p)^s u = |u|^{p_s^*-2} u+\lambda |u|^{q-2}u \quad \text{in} \,\, \Omega,\\&
u=0 \quad \text{in} \,\, \mathbb{R}^N\setminus \Omega,
\end{cases}
\end{equation}
where $\lambda>0$, $0<s<1$, $\Omega\subset \mathbb{R}^N$ is a bounded domain, $ps<N$,  $1<q<p_s^*:= \frac{pN}{N-sp}$. The fractional $p$- Laplacian $(-\Delta_p)^s u$ is defined as 
\begin{equation*}
    (-\Delta_p)^su(x):= C_{N,p,s} \lim _{\varepsilon \rightarrow 0} \int_{\mathbb{R}^{N} \backslash B_{\varepsilon}(x)} \frac{|u(x)-u(y)|^{p-2}(u(x)-u(y))}{|x-y|^{N+s p}} d y,\quad x\in \mathbb{R}^N,
\end{equation*}
where 
\begin{equation}\label{def c}
C_{N,p,s}:=\frac{\frac{sp}{2}(1-s)2^{2s-1}}{\pi^{\frac{N-1}{2}}}\frac{\Gamma(\frac{N+ps}{2})}{\Gamma(\frac{p+1}{2})\Gamma(2-s)}
\end{equation}
is the normalizing constant \cite{DGV21}. Servadei and Valdinoci \cite{SV2013, SV2015} studied the existence of positive solutions when $p=q=2$ and for all $s\in (0,1)$ whereas Tan \cite{T2011} studied a fractional Brezis-Nirenberg type problem involving the square root of the Laplacian. Mosconi {\it et al.} \cite{MPSY16} studied the problem \eqref{P2} in the quasilinear case involving fractional $p$-Laplacian when $q=p$. They proved the existence of nontrivial solutions extending the works due to \cite{Degiovanni2007, Egnell1988, AP1991} to the fractional case. In \cite{LST2017}, the authors established the existence of infinitely many solutions for $p=q=2$ using the Caffarelli-Silvestre extension, where the spectral version of the fractional Laplacian was considered. Ye and Zhang \cite{YZ2024} proved that for $1<q<p$, there exist $\lambda_*$ and $\lambda^*$, such that for all $\lambda\in(0,\lambda_*)$, the problem \eqref{P2} has a positive solution whereas for all $\lambda\in(0,\lambda^*)$, the problem \eqref{P2} has infinitely many solutions.

Recently, Brezis–Nirenberg type problems involving mixed local and nonlocal operators have attracted significant attention, and are given by
\begin{equation} \label{P3}
\begin{cases}
&-\Delta_pu +(-\Delta_p)^s u = |u|^{p^*-2} u+\lambda |u|^{q-2}u \quad \text{in} \,\, \Omega,\\&
u=0 \quad \text{in} \,\, \mathbb{R}^N\setminus \Omega,
\end{cases}
\end{equation}
where $\lambda>0$, $s\in(0,1)$, $\Omega\subset \mathbb{R}^N$ is a bounded domain, $1<p<N$ and $1<q<p^*:= \frac{pN}{N-p}$. When $p=2$, Biagi {\it et al.} \cite{BDVV22} studied the problem \eqref{P3} extending the works of Brezis-Nirenberg \cite{BN83} and Servadei-Valdinoci \cite{SV2013, SV2015}.  Da Silva \textit{et al.} \cite{DFV24} extended it for the whole range of $p\in(1,\infty)$ by exploring the results in three cases: $q\in(1,p)$, $q=p$ and $q\in(p,p^*)$. They establish the existence of a positive constant $\lambda_*>0$ such that for any $\lambda \in (0,\lambda_*)$, the problem \eqref{P3} possesses infinitely many nontrivial solutions with negative energy for $1<q<p$.  Moreover, if a certain condition is satisfied, they show the existence of multiple nontrivial solutions to the problem \eqref{P3} when $p=q$.  Furthermore, if a certain condition occurs, it proves that problem \eqref{P3} possesses a nontrivial solution when $p<q<p^*$. Recently, the first two authors \cite{BG2025} established the existence of infinitely many sign-changing solutions to a similar problem as \eqref{P3} with subcritical nonlinearities for $1<p<\infty$. For further studies involving mixed local and nonlocal operators, we refer to \cite{BG2025, BDVV22, DFV24, GCF2025} and the references therein. 

In this paper, we deal with the superposition operator $A_{\mu, p}$ of nonlinear fractional operators, introduced by Dipierro \textit{et al.} \cite{DPSV, DPSV2} and defined as follows:
\begin{equation}\label{superposition operator}
    A_{\mu, p} u:=\int_{[0,1]}(-\Delta)_{p}^{s} u \d \mu(s),
\end{equation} 
where $\mu$ is a signed measure given by
\begin{equation} \label{mudef}
    \mu:=\mu^{+}-\mu^{-}. 
\end{equation}
Here, the measures $\mu^{+}$ and $\mu^{-}$ are two nonnegative finite Borel measures on $[0,1]$ holds the following conditions: 
\begin{equation}\label{measure 1}
    \mu^{+}([\bar{s}, 1])>0,
\end{equation}
\begin{equation}\label{measure 2}
    \mu^{-}|_{[\bar{s}, 1]}=0,
\end{equation}
and
\begin{equation}\label{measure 3}
\mu^{-}([0, \bar{s}]) \leq \kappa \mu^{+}([\bar{s}, 1]) ,
\end{equation}
for some $\bar{s} \in(0,1]$ and $\kappa \geq 0.$ 

Note that, utilizing the condition \eqref{measure 1}, there exists another fractional exponent $s_{\sharp} \in[\bar{s}, 1]$ such that 
\begin{equation}\label{measure 4}
\mu^{+}\left(\left[s_{\sharp}, 1\right]\right)>0. 
\end{equation}

The notation $(-\Delta)_p^s$ denotes the fractional $p$-Laplacian, with the constant $C_{N,p,s}$ chosen to guarantee consistent limits as $s \nearrow 1$ and as $s \searrow 0$, specifically:
\begin{align*}
& \lim _{s \searrow 0}(-\Delta)_{p}^{s} u=(-\Delta)_{p}^{0} u:={|u|^{p-2}}u, \\
& \lim _{s \nearrow 1}(-\Delta)_{p}^{s} u=(-\Delta)_{p}^{1} u:=-\Delta_{p} u=-\operatorname{div}\left(|\nabla u|^{p-2} \nabla u\right).
\end{align*}

The signed measure $\mu$ allows each operator $(-\Delta)_p^s$ in \eqref{superposition operator} to influence the definition with potentially different signs. In our setting, conditions \eqref{measure 1}–\eqref{measure 3} facilitate the division of the interval $[0, 1]$ into two subintervals, where the right subinterval carries a specific sign and sufficiently outweighs the influence of the left. Since $\mu$ is a signed measure, each fractional $p$-Laplacian appearing in \eqref{superposition operator} may contribute with distinct signs. However, the assumptions \eqref{measure 1}–\eqref{measure 3} ensure that the negative parts of $\mu$ are effectively counterbalanced by its positive parts. As a result, the negative contributions in \eqref{superposition operator} do not impact the behaviour associated with higher fractional exponents.

An intriguing feature of the superposition operator in \eqref{superposition operator} is to include the negative $p$-Laplacian $-\Delta_{p} u$ (when $\mu$ is the Dirac measure concentrated at 1), the fractional $p$-Laplacian $(-\Delta)_p^{s}$ (when $\mu$ is the Dirac measure concentrated at a fractional power $s \in (0, 1)$)  and a  ``mixed order operator" $-\Delta_p+\epsilon (-\Delta)_p^{s}$ where $\epsilon \in (0,1]$ (when $\mu$ is the sum of two Dirac measures $\delta_1+\epsilon \delta_s,\,s \in (0, 1)$), among others. A noteworthy feature of the operators is that certain components of the operator may carry the “wrong sign,” as long as a dominant contribution typically associated with operators of higher fractional order ensures overall control. In the final section, that is, Section \ref{secexpa}, we present several concrete examples illustrating these phenomena and demonstrate the applicability of our main results. These superposition operators have several significant applications in anomalous diffusion, population dynamics, and mathematical biology, such as in models involving Gaussian processes and L\'evy flights. We refer to \cite{DV21, DPSV25, DPSV2025} for more discussion on this topic.

Initially, when $ p=2 $, Dipierro \textit{et al.} \cite{DPSV2} study the superposition problem involving jump nonlinearity and critical exponent.  Subsequently, for $p\in(1,\infty)$, Dipierro \textit{et al.} \cite{DPSV} examined the following linear problem involving critical exponent:
\begin{align}\label{problem 1}
A_{\mu,p} u &= \lambda |u|^{p-2}u+|u|^{p^*_{s_\sharp}-2}u ~\text{in}~\Omega,  \nonumber\\
 u & =0 ~\text{in}~ \mathbb{R}^N \setminus \Omega,
 \end{align}
where $\Omega$ be a bounded open subset of $\mathbb{R}^{N}$, $0\leq s\leq 1<p<\frac{N}{s_\sharp}$ and the critical exponent $p_{s_{\sharp}}^{*}$  defined as
\begin{equation}\label{critical exponent}
    p_{s_{\sharp}}^{*}:=\frac{N p}{N-s_{\sharp} p}.
\end{equation}
They prove that, there exists $\kappa_0>0$ such that if $\kappa \in [0,\kappa_0]$, the problem \eqref{problem 1} possesses $m$ pair of distinct nontrivial solutions when $\lambda$ satisfies $$\lambda_l-\frac{\mathcal{S}(p)}{|\Omega|^{\frac{s_\sharp p}{N}}}<\lambda<\lambda_l=....=\lambda_{l+m-1},$$
for some $l,m\geq 1$ and $\mathcal{S}(p)$ is defined in \eqref{eq3.5}. Very recently, the last two authors with Aikyn and Ruzhansky  \cite{AGKR2025} studied the following superlinear problem involving critical exponent:
\begin{align}\label{problem 3}
A_{\mu,p} u &= \lambda |u|^{q-2}u+|u|^{p^*_{s_\sharp}-2}u ~\text{in}~\Omega,  \nonumber\\
 u & =0 ~\text{in}~ \mathbb{R}^N \setminus \Omega,
 \end{align}
where $\Omega$ is a bounded open subset of $\mathbb{R}^{N}$, $0<\lambda$, $0\leq s\leq 1<p<\frac{N}{s_\sharp}$ and $p<q<p^*_{s_\sharp}$. They showed the existence of a positive constant $\lambda^*$ such that for any $\lambda \geq \lambda^*$, there exists a sufficiently small $\kappa_0>0$ such that for $\kappa \in [0, \kappa_0]$, problem \eqref{problem 3} possesses at least one nontrivial solution. For further study in this area, we suggest referring to the following papers \cite{ABB2025,AGKR2025,DV21,DPSV,DPSV2,DPSV1,DPSV25,DPSV2025} and the references therein.

Inspired by the aforementioned papers, particularly \cite{DFV24, DPSV}, we examine the following sublinear problem involving the critical exponent:
\begin{align}\label{problem 2}
A_{\mu,p} u &= \lambda |u|^{q-2}u+|u|^{p^*_{s_\sharp}-2}u ~\text{in}~\Omega  \nonumber\\
 u & =0 ~\text{in}~ \mathbb{R}^N \setminus \Omega,
 \end{align}
where $\Omega$ is a bounded open subset of $\mathbb{R}^{N}$, $0<\lambda$, $0\leq s\leq 1<p<\frac{N}{s_\sharp}$ and $1<q<p<p^*_{s_\sharp}$.

The primary difficulty in dealing with \eqref{problem 2} arises from the presence of the Sobolev critical exponent and the failure of the $(\mathrm{PS})_c$ condition for arbitrary energy levels $c\in \mathbb{R}$. The main challenges lie in verifying the $(\mathrm{PS})_c$ condition for the associated energy functional up to a certain level $c$, and in handling the superposition operator. The critical exponent prevents the use of compact embeddings, rendering standard techniques such as the convex-compact method and related approaches ineffective. Moreover, the $(\mathrm{PS})_c$ condition has not been previously established in the context of superposition operators when $1<q<p$, and the existing literature on such operators is quite limited. To overcome these obstacles, we adopt a variational framework that combines a truncation argument with the notion of genus, leading to the establishment of our main result.

\begin{thm}\label{thm1.1} Let $\Omega$ be a bounded open subset of $\mathbb{R}^{N}.$   Let $\mu=\mu^{+}-\mu^{-}$ with $\mu^{+}$ and $\mu^{-}$ satisfy conditions \eqref{measure 1}-\eqref{measure 3}. Assume that $s_{\sharp}$ is given as in \eqref{measure 4}.
   Let $1 < p < \frac{N}{s_\sharp}$ and $1 < q < p$. Then, there exist $\lambda_*>0$ and $\kappa_* \geq 0$ such that for all $\lambda \in (0, \lambda_*)$ and for all $\kappa \in [0, \kappa_*]$, the problem \eqref{problem 2} possesses infinitely many nontrivial solutions in $X_p(\Omega)$ with negative energy.
\end{thm}

We emphasize here that Theorem \ref{thm1.1}  complements the existence results for the Brezis--Nirenberg-type problems obtained by Dipierro {\it et al.} \cite{DPSV} for $q = p$ and by Aikyn {\it et al.} \cite{AGKR2025} for $p < q < p_{s_\sharp}^*$. It is worth mentioning that our approach extends and generalizes earlier results by Garc{\'i}a Azorero and Peral Alonso \cite{AP1991} and by Da Silva, Fiscella, and Viloria \cite{DFV24}.

The remainder of the paper is organized as follows. In Section \ref{pre}, we introduce the functional framework for our problem, review key preliminary results, and define an equivalent energy functional to the norm in the associated Sobolev space. Section \ref{sec7} is devoted to proving the main result, namely Theorem \ref{thm1.1}, by defining the notion of weak solutions and establishing foundational results tailored to our setting. Finally, in Section \ref{secexpa}, we illustrate the applicability of our main theorem through several concrete and meaningful examples.

\section{Functional Analytic setting for the superposition operator} \label{pre}

This section is devoted to the construction of the appropriate notation of the fractional Sobolev spaces and their properties, which are crucial for analyzing our problem. Similar developments of such spaces can be found in \cite{AGKR2025, DPSV, DPSV2, DPSV1}. Initially, for $s \in [0, 1]$, we define
\begin{equation*}
    [u]_{s, p}:= \begin{cases}\|u\|_{L^{p}\left(\mathbb{R}^{N}\right)} & \text { if } s=0, \\ \left(C_{N, s, p} \iint_{\mathbb{R}^{2 N}} \frac{|u(x)-u(y)|^{p}}{|x-y|^{N+s p}} \d x \d y\right)^{1 / p} & \text { if } s \in(0,1), \\ \|\nabla u\|_{L^{p}\left(\mathbb{R}^{N}\right)} & \text { if } s=1,\end{cases}
\end{equation*}
where $C_{N, s, p}$ is the normalizing constant as in \eqref{def c} and we have
$$\lim _{s \searrow 0^+}[u]_{s, p}=[u]_{0, p} \quad \text { and } \quad \lim _{s \nearrow 1^-}[u]_{s, p}=[u]_{1, p}.$$

We define the space $X_{p}(\Omega)$ as the collection of measurable functions $u: \mathbb{R}^{N} \rightarrow \mathbb{R}$ such that $u=0$ in $\mathbb{R}^{N} \backslash \Omega$, for which the following norm is finite:
\begin{equation}\label{norm on Xp} \|u\|_{X_p(\Omega)} = \rho_{p}(u):=\left(\int_{[0,1]}[u]_{s, p}^{p} \d \mu^{+}(s)\right)^{1 / p}< +\infty.
\end{equation}
The space $X_p(\Omega)$ was initially introduced in \cite{DPSV2} for $p=2$ and subsequently generalized in \cite{DPSV} for $1<p<\infty$. We present the following results from \cite{DPSV} (see also \cite{AGKR2025}), which are necessary for our study.
\begin{lem}\cite{AGKR2025, DPSV}\label{Uniform convexity}
    The $X_{p}(\Omega)$ is a separable Banach space for $1 \leq p < \infty$ and it is a uniformly convex Banach space, for $1<p < \infty$.
\end{lem}

\begin{lem}\label{reabsorb} \cite[Proposition 4.1]{DPSV}
    Let $1<p<N$ and \eqref{measure 2} and \eqref{measure 3} be true. Then, there exists $c_{0}=c_{0}(N, \Omega, p)>0$ such that, for any $u \in X_{p}(\Omega)$, we have
\begin{equation*}
    \int_{[0, \bar{s}]}[u]_{s, p}^{p} \d \mu^{-}(s) \leq c_{0} \kappa \int_{[\bar{s}, 1]}[u]_{s, p}^{p} \d \mu(s)=c_{0} \kappa \int_{[\bar{s}, 1]}[u]_{s, p}^{p} \d \mu^+(s).
\end{equation*}
\end{lem}
The following uniform Sobolev embedding lemma is essential to establish the embeddings for the space $X_p(\Omega)$ into Lebesgue spaces.
\begin{lem}\label{Sobolev emb} \cite[Theorem 3.2]{DPSV}
Let $\Omega$ be an open, bounded subset of $\mathbb{R}^{N}$ and $p \in(1, N)$.
Then, there exists $C=C(N, \Omega, p)>0$ such that, for every $s_1, s_2 \in[0,1]$ with $s_1 \leq s_2$ and every measurable function $u: \mathbb{R}^{N} \rightarrow \mathbb{R}$ with $u=0$ a.e. in $\mathbb{R}^{N} \setminus \Omega$, we have
$$
[u]_{s_1, p} \leq C[u]_{s_2, p}.
$$
\end{lem} 
\begin{prop}\cite[Proposition 2.5]{AGKR2025}\label{compact and cont embedding}
Assume that \eqref{measure 1}–\eqref{measure 4} hold. Then, there exists a positive constant $\bar{c}=\bar{c}\left(N, \Omega, s_{\sharp}, p\right)$ such that, for any $u \in X_{p}(\Omega)$, we have
\begin{equation} \label{Xpusialemb}
[u]_{s_{\sharp},p} \leq \bar{c}\left(\int_{[0,1]}[u]_{s,p}^{p} \mathrm{~d} \mu^{+}(s)\right)^{\frac{1}{p}}.
\end{equation}
 In particular, the embedding 
\begin{equation}\label{embedding}
    X_p(\Omega) \hookrightarrow L^{r}(\Omega)
\end{equation}
is continuous for all $r \in[1,p_{s_{\sharp}}^{*}]$ and compact for all $r \in[1,p_{s_{\sharp}}^{*})$.
\end{prop}

We now recall the following weak convergence results from \cite{DPSV}. 
\begin{lem} \label{weak convergence}  \cite[ Lemma 5.8]{DPSV}
    Let $(u_{n})$ be a bounded sequence in $X_{p}(\Omega)$. Then, there exists $u \in X_{p}(\Omega)$ such that for all $v \in X_{p}(\Omega)$, we have
\begin{align}
    \lim _{n \rightarrow\infty} &\int_{[0,1]}\left(\iint_{\mathbb{R}^{2 N}} \frac{\left|u_{n}(x)-u_{n}(y)\right|^{p-2}\left(u_{n}(x)-u_{n}(y)\right)(v(x)-v(y))}{|x-y|^{N+s p}} \d x \d y\right) \d \mu^{ \pm}(s)\nonumber \\
&=\int_{[0,1]}\left(\iint_{\mathbb{R}^{2 N}} \frac{|u(x)-u(y)|^{p-2}(u(x)-u(y))(v(x)-v(y))}{|x-y|^{N+s p}} \d x \d y\right) \d \mu^{ \pm}(s).
\end{align}
\end{lem}
We state the following Brezis-Lieb type lemma from \cite{DPSV}. 
\begin{lem}\label{B-L lemma} \cite[ Lemma 5.9]{DPSV}
   Let $(u_{n})$ be a bounded sequence in $X_{p}(\Omega)$. Suppose that $u_{n}$ converges to some $u$ a.e. in $\mathbb{R}^{N}$ as $n \rightarrow\infty$. Then we have
\begin{equation}
\int_{[0,1]}[u]_{s, p}^{p} \d \mu^{ \pm}(s)=\lim _{n \rightarrow\infty}\left(\int_{[0,1]}\left[u_{n}\right]_{s, p}^{p} \d \mu^{ \pm}(s)-\int_{[0,1]}\left[u_{n}-u\right]_{s, p}^{p} \d \mu^{ \pm}(s)\right).
\end{equation}
\end{lem}

We now construct a new energy functional $\eta_{p}$ on $X_p(\Omega)$ as 
 \begin{align}\label{eq2.6}
         \eta_{p}(u):=\left(\int_{[0,1]}[u]_{s, p}^{p} \d \mu^{+}(s)-\int_{[0,\bar{s})}[u]_{s, p}^{p} \d \mu^{-}(s)\right)^{1 / p}.
    \end{align}
Note that for $\kappa=0$, $\eta_{p}$ is a norm on $X_p(\Omega)$ and we have $\rho_p(u)=\eta_p(u)$ for any $u\in X_p(\Omega)$. Moreover, when $p=2$,  $\eta_2(u)$ acts as a norm on $X_2(\Omega)$. Indeed, we define $\eta_2(u):=\sqrt{\langle u,u \rangle}$ for any $u\in X_2(\Omega)$ with $${\langle u,v\rangle}:={\langle u,v\rangle}^+-{\langle u,v\rangle}^-,$$ 
    where
    $${\langle u,v\rangle}^+:=\int_{[0,1]}\left(\iint_{\mathbb{R}^{2 N}} \frac{C_{N, s}(u(x)-u(y))(v(x)-v(y))}{|x-y|^{N+2s}} \d x \d y\right) \d \mu^{+}(s),$$
    $${\langle u,v\rangle}^-:=\int_{[0, \bar{s}]}\left(\iint_{\mathbb{R}^{2 N}} \frac{C_{N, s}(u(x)-u(y))(v(x)-v(y))}{|x-y|^{N+2s}} \d x \d y\right) \d \mu^{-}(s).$$
    Now, using Cauchy–Schwarz inequality, we get
    \begin{align*}
        (\eta_2(u_1+u_2))^2 =&\langle u_1+u_2, u_1+u_2 \rangle =  (\eta_2(u_1))^2 + (\eta_2(u_2))^2 +2 \langle u_1,u_2 \rangle \\
        \leq & (\eta_2(u_1))^2 + (\eta_2(u_2))^2 +2 |\langle u_1,u_2 \rangle| \\
        \leq & (\eta_2(u_1))^2 + (\eta_2(u_2))^2 +2 \eta_2(u_1)\eta_2(u_2)=(\eta_2(u_1)+\eta_2(u_2))^2.
    \end{align*}
    Therefore, $X_2(\Omega)$ is a real inner product space with inner product $\langle.,.\rangle$, and $\eta_2$ is the induced norm. Thus, for any $\kappa \in [0,\kappa_0]$, $X_2(\Omega)$ reduces to a Hilbert space concerning the inner product $\langle u,v \rangle$, where $\kappa_0$ is sufficiently small. For further properties of the Hilbert space $X_2(\Omega)$ and related problems, we refer to \cite{ABB2025, AGKR2025,  DPSV2025}. 

The next lemma guarantees that $\eta_p$ and $\rho_p$ are equivalent as energy functional on $X_p(\Omega)$ for all $p\in[1,\infty)$.

    \begin{lem}\label{lmn2.8}
        Let $\mu$ satisfies the assumptions \eqref{measure 1}-\eqref{measure 4}. Then there exists a sufficiently small $\kappa_0(N,\Omega,p)$ such that for all $\kappa\in [0,\kappa_0]$, the functional $\eta_p$ as in \eqref{eq2.6} and the norm $\rho_p$ as in \eqref{norm on Xp} on $X_p(\Omega)$ are equivalent for all $p\in[1,\infty)$. In particular, $\eta_p$ is a quasi-norm.
    \end{lem}
\begin{proof}
    Thanks to Lemma \ref{reabsorb} that for a sufficiently small choice of $\kappa$, we get $(1-c_0\kappa)>0$ and 
    \begin{align}\label{eq2.7}
        \left(\int_{[0,1]}[u]_{s, p}^{p} \d \mu^{+}(s)-\int_{[0,\bar{s})}[u]_{s, p}^{p} \d \mu^{-}(s)\right) \geq (1-c_0\kappa) \int_{[0,1]}[u]_{s, p}^{p} \d \mu^{+}(s)\geq 0.
    \end{align}
    The above inequality \eqref{eq2.7} asserts that the functional $\eta_p$ is well-defined. Using \eqref{eq2.7}, we claim that $\eta_p$ satisfy the following conditions on $X_p(\Omega)$. 
    \begin{enumerate}
        \item [$(i)$] $\eta_p(u)\geq 0$ and $\eta_p(u)= 0$ if and only if $u=0$,
        \item [$(ii)$] $\eta_p(\alpha u)=|\alpha| \eta_p(\alpha u),$ $\alpha$ is scalar.
    \end{enumerate} 
    In addition, for every $u\in X_p(\Omega)$, we obtain
    \begin{align}\label{equi}
        (1-c_0k)^\frac{1}{p} \rho_p(u)\leq \eta_p (u) \leq \rho_p(u).
    \end{align}
    Thus, we conclude that $\rho_p$ and $\eta_p$ are equivalent functionals on $X_p(\Omega)$. Since $\rho_p$ is a norm, using \eqref{equi}, we conclude that $\eta_p$ is a quasi-norm. This completes the proof.
\end{proof}

\begin{rem}
Recall that $\eta_p,\, (p\neq2)$ is a norm for $k=0$ or $p=2$. It is noteworthy to mention that it is challenging to prove whether $\eta_p,\, (p\neq2)$ is a norm involving the conditions \eqref{measure 1}-\eqref{measure 4}. However, the equivalence as energy functionals of  $\eta_p$ and $\rho_p$ in $X_p(\Omega)$ gives us advantages for obtaining a priori estimates for the functional $\mathcal{J}_\lambda$. Specifically, Lemma \ref{lmn2.8} helps proving the coercivity of the truncated functional $\mathcal{J}_\lambda$ (Sec 3).
\end{rem}

\section{Existence of infinitely many solutions} \label{sec7}

This section establishes the main result obtaining the fundamental tools related to the symmetric mountain pass theorem for the problem \eqref{problem 2} involving the elliptic operator $A_{\mu, p}$. We first define the notion of a weak solution to the problem \eqref{problem 2}.

\begin{definition}\label{weaksol MP}
    We say $u \in X_{p}(\Omega)$ is a weak solution to the problem \eqref{problem 2} if for all $v \in X_{p}(\Omega)$, we have
\begin{align*}
& \int_{[0,1]}\left(\iint_{\mathbb{R}^{2 N}} \frac{C_{N, s, p}|u(x)-u(y)|^{p-2}(u(x)-u(y))(v(x)-v(y))}{|x-y|^{N+s p}} \d x \d y\right) \d \mu^{+}(s) \\
& =\int_{[0, \bar{s}]}\left(\iint_{\mathbb{R}^{2 N}} \frac{C_{N, s, p}|u(x)-u(y)|^{p-2}(u(x)-u(y))(v(x)-v(y))}{|x-y|^{N+s p}} \d x \d y\right) \d \mu^{-}(s) \\
& \quad+\lambda \int_{\Omega}|u|^{q-2} u v \d x+\int_{\Omega}|u|^{p_{s_\sharp}^{*}-2} u v \d x.
\end{align*}
\end{definition}
The weak solutions to the problem \eqref{problem 2} are the critical points of the energy functional $\mathcal{I}_{\lambda}: X_{p}(\Omega) \rightarrow \mathbb{R}$, which is defined as
\begin{equation}\label{functional MP}
\mathcal{I}_{\lambda}(u):=\frac{1}{p}\left(\rho_{p}(u)\right) ^{p}-\frac{1}{p} \int_{[0, \bar{s}]}[u]_{s, p}^{p} \d \mu ^{-}(s)-\frac{\lambda}{q} \int_{\Omega}|u|^{q} \d x-\frac{1}{p_{s_{\sharp}}^{*}} \int_{\Omega}|u|^{p_{s_\sharp}^{*}} \d x.
\end{equation}
Furthermore, we have
\begin{align*}
\langle \mathcal{I}^{\prime}_{\lambda}(u), v\rangle&= \int_{[0,1]}\left(\iint_{\mathbb{R}^{2 N}}  \frac{C_{N, s, p}|u(x)-u(y)|^{p-2}(u(x)-u(y))(v(x)-v(y))}{|x-y|^{N+s p}} \d x \d y\right) \d \mu^{+}(s) \\
& -\int_{[0, \bar{s}]}\left(\iint_{\mathbb{R}^{2 N}} \frac{C_{N, s, p}|u(x)-u(y)|^{p-2}(u(x)-u(y))(v(x)-v(y))}{|x-y|^{N+s p}} \d x \d y\right) \d \mu^{-}(s) \\
& -\lambda \int_{\Omega}|u|^{q-2}u v \d x- \int_{\Omega}|u|^{p_{s_\sharp}^{*}-2} u v \d x,
\end{align*}
for all $u, v \in X_{p}(\Omega)$.
\begin{rem}
    We mention here that the integral on the LHS in Definition \ref{weaksol MP} has the following explicit form: 
\begin{align*} 
&\int_{[0,1]}\left(\iint_{\mathbb{R}^{2 N}} \frac{C_{N, s, p}|u(x)-u(y)|^{p-2}(u(x)-u(y))(v(x)-v(y))}{|x-y|^{N+s p}} \d x \d y\right) \d \mu^{+}(s)\\
&=\int_{(0,1)}\left(\iint_{\mathbb{R}^{2 N}} \frac{C_{N, s, p}|u(x)-u(y)|^{p-2}(u(x)-u(y))(v(x)-v(y))}{|x-y|^{N+s p}} \d x \d y\right) \d \mu^{+}(s)\\ 
&+\mu^+(\{0\}) \int_\Omega |u(x)|^{p-2}u(x)v(x) \d x + \mu^+(\{1\}) \int_\Omega |\nabla u(x)|^{p-2}\nabla u(x) \cdot \nabla v(x) \d x. \end{align*} 
Similarly, the first term on the RHS is defined but for simplicity, we will continue to misuse the expression as mentioned in Definition \ref{weaksol MP}.
\end{rem}

The next lemma establishes the Palais-Smale $(\mathrm{PS})_c$ condition for the energy functional $\mathcal{I}_{\lambda}$ at level $c$ for the full range of $q\in(1,p^*_{s_{\sharp}})$. Note that it is sufficient to obtain the $\mathrm{PS}$ condition for $1<q<p$ for the problem \eqref{problem 2}. The result for the range $q\in[p, p^*_{s_{\sharp}})$ is of general interest to study similar problems with $p$-superlinear growth. The idea of the proof is analogous to \cite[Proposition 5.10]{DPSV} and \cite[Lemma 7.5]{AGKR2025}. Thanks to the continuous Sobolev embedding $X_{p}(\Omega) \hookrightarrow L^{p_{s_\sharp}^{*}}(\Omega)$ given by \eqref{embedding}, we define the best Sobolev constant as
\begin{equation}\label{eq3.5}
    \mathcal{S}(p):=\inf_{\|u\|_{L^{p_{s_\sharp}^{*}}(\Omega)}=1} \int_{[0,1]}[u]_{s, p}^{p} d \mu^{+}(s).
    \end{equation}
 \begin{lem}\label{PS condition lem}
    Let $\theta_{0} \in(0,1)$ and
\begin{align}\label{eq3.4}
~&(i)~c^{*}:=\frac{s_{\sharp}}{N}\left(\left(1-\theta_{0}\right) \mathcal{S}(p)\right)^{N / s_{\sharp} p}-|\Omega|\left(\frac{s_\sharp}{N}\right)^{-\frac{q}{p^*_{s_{\sharp}}-q}}\left[\lambda \left(\frac{1}{q}-\frac{1}{p}\right)\right]^{\frac{p^*_{s_{\sharp}}}{p^*_{s_{\sharp}}-q}} ~\text{when}~ q\in(1,p), \\
~&(ii)~c^{*}:=\frac{s_{\sharp}}{N}\left(\left(1-\theta_{0}\right) \mathcal{S}(p)\right)^{N / s_{\sharp} p}~\text{when}~ q\in[p,p^*_{s_{\sharp}}),\label{bound for level}
\end{align}
Then, there exists $\kappa_{0}(N, \Omega, p, s_{\sharp},\theta_{0})>0$, such that for any $\kappa \in\left[0, \kappa_{0}\right]$ and $c \in$ $\left(0, c^{*}\right)$, the functional $\mathcal{I}_\lambda$ defined in \eqref{functional MP} satisfies the $\mathrm{PS}$ condition at level $c.$
\end{lem}

\begin{proof}
    Let $c \in\left(0, c^{*}\right)$ and $(u_{n}) \subset X_{p}(\Omega)$ be a $\mathrm{PS}$ sequence for the functional $\mathcal{I}_{\lambda}$. Therefore, we get 
\begin{align}\label{level c}
\begin{split}
    \lim _{n \rightarrow\infty} \mathcal{I}_{\lambda}\left(u_{n}\right) =&\lim _{n \rightarrow\infty} \Bigg[\frac{1}{p}\left(\rho_{p}\left(u_{n}\right)\right)^{p}-\frac{1}{p} \int_{[0, \bar{s}]}\left[u_{n}\right]_{s, p}^{p} \d \mu^{-}(s)\\
    &-\frac{\lambda}{q} \int_{\Omega}\left|u_{n}\right|^{q} \d x-\frac{1}{p_{s_{\sharp}}^{*}} \int_{\Omega}\left|u_{n}\right|^{p_{s_{\sharp}}^{*}} \d x \Bigg]=c
\end{split}
\end{align}
and
\begin{align}\label{derivative}
 \lim _{n \rightarrow\infty}& \sup_{v \in X_p(\Omega)} \langle \mathcal{I}^{\prime}_{\lambda}\left(u_{n}\right), v\rangle \nonumber \\
=&  \lim _{n \rightarrow\infty} \sup_{v \in X_p(\Omega)} \Bigg[ \int_{[0,1]}\Bigg( \iint_{\mathbb{R}^{2 N}}  \frac{\splitfrac{C_{N, s, p}\left|u_{n}(x)-u_{n}(y)\right|^{p-2}\left(u_{n}(x)-u_{n}(y)\right)}{\times(v(x)-v(y))}}{|x-y|^{N+s p}}\d x \d y\Bigg) \d \mu^{+}(s) \nonumber\\ 
& -\int_{[0, \bar{s}]}\Bigg( \iint_{\mathbb{R}^{2 N}} \frac{\splitfrac{C_{N, s, p}\left|u_{n}(x)-u_{n}(y)\right|^{p-2}\left(u_{n}(x)-u_{n}(y)\right)}{\times(v(x)-v(y))}}{|x-y|^{N+s p}} \d x \d y\Bigg) \d \mu^{-}(s)\nonumber \\
& -\lambda \int_{\Omega}\left|u_{n}\right|^{q-2} u_{n} v \d x-\int_{\Omega}\left|u_{n}\right|^{p_{s_\sharp}^{*}-2} u_{n} v \d x\Bigg]=0. 
\end{align}
Choosing $v := -u_{n}$ in \eqref{derivative}, we derive
\begin{align*}
0  \leq &\lim _{n \rightarrow\infty} \Bigg[\int_{[0,1]}\left(C_{N, s, p} \iint_{\mathbb{R}^{2 N}} \frac{\left|u_{n}(x)-u_{n}(y)\right|^{p}}{|x-y|^{N+s p}} \d x \d y\right) \d \mu^{+}(s) \\
& -\int_{[0, \bar{s}]}\left(C_{N, s, p} \iint_{\mathbb{R}^{2 N}} \frac{\left|u_{n}(x)-u_{n}(y)\right|^{p}}{|x-y|^{N+s p}} \d x \d y\right) \d \mu^{-}(s) \\
& -\lambda \int_{\Omega}\left|u_{n}\right|^{q} \d x-\int_{\Omega}\left|u_{n}\right|^{p_{s_\sharp}^{*}} \d x\Bigg] \\
 =&  \lim _{n \rightarrow\infty} p \mathcal{I}_{\lambda}\left(u_{n}\right)+\lim _{n \rightarrow\infty} \Bigg[\lambda\left(\frac{p}{q}-1\right) \|u_n\|_{L^q(\Omega)}^q + \left( \frac{p}{p_{s_\sharp}^{*}}-1\right)\|u_n\|_{L^{p_{s_\sharp}^{*}}(\Omega)}^{p_{s_\sharp}^{*}} \Bigg].
\end{align*}
Thus, from \eqref{level c}, we get
\begin{equation}
\lim _{n \rightarrow\infty} \Bigg[\lambda\left(\frac{1}{p}-\frac{1}{q}\right) \|u_n\|_{L^q(\Omega)}^q + \left( \frac{1}{p}-\frac{1}{p_{s_\sharp}^{*}}\right)\|u_n\|_{L^{p_{s_\sharp}^{*}}(\Omega)}^{p_{s_\sharp}^{*}}\Bigg] \leq c.
\end{equation}
Also, from Lemma \ref{reabsorb}, we obtain
\begin{equation*}
    \int_{[0, \bar{s}]}\left[u_{n}\right]_{s, p}^{p} \d \mu^{-}(s) \leq c_{0} \kappa \int_{[\bar{s}, 1]}\left[u_{n}\right]_{s, p}^{p} \d \mu^{+}(s) \leq c_{0} \kappa\left(\rho_{p}\left(u_{n}\right)\right)^{p},
\end{equation*}
for some $c_0:=c_0(N, \Omega, p).$ Therefore, we get
\begin{equation}\label{main ineq in proof}
\left(\rho_{p}\left(u_{n}\right)\right)^{p}-\int_{[0, \bar{s}]}\left[u_{n}\right]_{s, p}^{p} \d \mu^{-}(s) \geq\left(1-c_0 \kappa\right)\left(\rho_{p}\left(u_{n}\right)\right)^{p}.
\end{equation}
Now, combining \eqref{main ineq in proof} and \eqref{level c}, we conclude that $\rho_{p}\left(u_{n}\right)$ is uniformly bounded in $n$, provided $1 - c_0 \kappa > 0$.  Thus, from Lemma \ref{Uniform convexity} and Proposition \ref{compact and cont embedding}, there exists $u \in X_{p}(\Omega)$ and a subsequence of $(u_n)$, still denoted by $(u_n)$ such that
\begin{align}\label{convergences}
\begin{split}
    & u_{n} \rightharpoonup u \text { in } X_{p}(\Omega) \\
& u_{n} \rightarrow u \text { in } L^{r}(\Omega) \text { for any } r \in [1, p_{s_{\sharp}}^{*}),  \\
& u_{n} \rightarrow u \text { a.e. in } \Omega .
\end{split}
\end{align}
We aim to prove that $u_{n} \rightarrow u$ strongly in $X_{p}(\Omega)$. Let us denote $\widetilde{u}_{n}:=u_{n}-u$. Now, using \cite[Theorem 1]{BL: 1983}, we have
\begin{equation}\label{ex eq 0}
    \|u\|_{L^{p_{s_\sharp}^{*}}(\Omega)}^{p_{s_\sharp}^{*}}=\lim _{n \rightarrow\infty}\left(\|u_n\|_{L^{p_{s_\sharp}^{*}}(\Omega)}^{p_{s_\sharp}^{*}}-\|\widetilde{u}_n\|_{L^{p_{s_\sharp}^{*}}(\Omega)}^{p_{s_\sharp}^{*}}\right).
\end{equation}
In addition, Lemma \ref{B-L lemma} implies that
\begin{equation}\label{ex eq 1}
\int_{[0,1]}[u]_{s, p}^{p} \d \mu^{ \pm}(s)=\lim _{n \rightarrow\infty} \left(\int_{[0,1]}\left[u_{n}\right]_{s, p}^{p} \d \mu^{ \pm}(s)-\int_{[0,1]}\left[\widetilde{u}_{n}\right]_{s, p}^{p} \d \mu^{ \pm}(s)\right).
\end{equation}
Now, the definition \ref{weaksol MP} with $v:=u$ yields
\begin{equation}\label{v=u}
\left(\rho_{p}(u)\right)^{p}-\int_{[0, \bar{s}]}[u]_{s, p}^{p} \d \mu^{-}(s)=\lambda\|u\|_{L^q(\Omega)}^{q}+\|u\|_{L^{p_{s_\sharp}^{*}}(\Omega)}^{p_{s_\sharp}^{*}}.
\end{equation}
Again, substituting $v:= \pm u_{n}$ into \eqref{derivative}, we get
\begin{align}\label{new sg}
     \lim _{n \rightarrow\infty} &\langle \mathcal{I}^{\prime}_{\lambda}\left(u_{n}\right), u_n\rangle \nonumber\\
     &=\lim _{n \rightarrow\infty} \left(\left(\rho_{p}(u_n)\right)^{p}-\int_{[0, \bar{s}]}[u_n]_{s, p}^{p} \d \mu^{-}(s)-\lambda\|u_n\|_{L^q(\Omega)}^{q}-\|u_n\|_{L^{p_{s_\sharp}^{*}}(\Omega)}^{p_{s_\sharp}^{*}}\right)=0.
\end{align}
From \eqref{new sg} and \eqref{convergences}, we get
\begin{equation}\label{ex eq 2}
\lim _{n \rightarrow\infty} \left(\left(\rho_{p}(u_n)\right)^{p}-\int_{[0, \bar{s}]}[u_n]_{s, p}^{p} \d \mu^{-}(s)-\|u_n\|_{L^{p_{s_\sharp}^{*}}(\Omega)}^{p_{s_\sharp}^{*}}\right)=\lambda\|u\|_{L^q(\Omega)}^{q}.
\end{equation}
Consequently, from \eqref{v=u}, we conclude that
\begin{align}\label{ex eq 00}
\begin{split}
    \lim _{n \rightarrow\infty} &\left(\left(\rho_{p}(u_n)\right)^{p}-\int_{[0, \bar{s}]}[u_n]_{s, p}^{p} \d \mu^{-}(s)-\|u_n\|_{L^{p_{s_\sharp}^{*}}(\Omega)}^{p_{s_\sharp}^{*}}\right)\\
&=\left(\rho_{p}(u)\right)^{p}-\int_{[0, \bar{s}]}[u]_{s, p}^{p} \d \mu^{-}(s)-\|u\|_{L^{p_{s_\sharp}^{*}}(\Omega)}^{p_{s_\sharp}^{*}}.
\end{split}
\end{align}
Thus, combining \eqref{ex eq 0},\eqref{ex eq 1} and \eqref{ex eq 00}, we obtain
\begin{equation}\label{eq3.15}
\lim _{n \rightarrow\infty} \left(\left(\rho_{p}(\widetilde{u}_n)\right)^{p}-\int_{[0, \bar{s}]}[\widetilde{u}_n]_{s, p}^{p} \d \mu^{-}(s)-\|\widetilde{u}_n\|_{L^{p_{s_\sharp}^{*}}(\Omega)}^{p_{s_\sharp}^{*}}\right)=0.
\end{equation}
Injecting the Sobolev constant as in \eqref{eq3.5} into \eqref{eq3.15}, we get
\begin{equation*}
    \lim _{n \rightarrow\infty}\left[\left(\rho_{p}\left(\widetilde{u}_{n}\right)\right)^{p}-\int_{[0, \bar{s}]}\left[\widetilde{u}_{n}\right]_{s, p}^{p} \d \mu^{-}(s)-\left(\frac{1}{\mathcal{S}(p)} \int_{[0,1]}\left[\widetilde{u}_{n}\right]_{s, p}^{p} \d \mu^{+}(s)\right)^{p_{s_\sharp}^{*} / p} \right]\leq 0,
\end{equation*}
which implies
\begin{equation*}
  \lim _{n \rightarrow\infty}\left[\left(\rho_{p}\left(\widetilde{u}_{n}\right)\right)^{p}-\int_{[0, \bar{s}]}\left[\widetilde{u}_{n}\right]_{s, p}^{p} \d \mu^{-}(s)-\frac{\left(\rho_{p}\left(\widetilde{u}_{n}\right)\right)^{p_{s_\sharp}^{*}}}{(\mathcal{S}(p))^{p_{s_\sharp}^{*} / p}}\right] \leq 0.  
\end{equation*}
Again, putting $\widetilde{u}_{n}$ instead of $u_{n}$ in \eqref{main ineq in proof}, we obtain
\begin{align}\label{main ineq 2}
    &\lim _{n \rightarrow\infty}\left[\left(1-c_0 \kappa\right)\left(\rho_{p}\left(\widetilde{u}_{n}\right)\right)^{p}-\frac{\left(\rho_{p}\left(\widetilde{u}_{n}\right)\right)^{p_{s_\sharp}^{*}}}{(\mathcal{S}(p))^{p_{s_\sharp}^{*} / p}}\right] \leq 0,\nonumber\\
    \Rightarrow&\lim _{n \rightarrow\infty}\left(\rho_{p}\left(\widetilde{u}_{n}\right)\right)^{p}\left[\left(1-c_0 \kappa\right)(\mathcal{S}(p))^{p_{s_\sharp}^{*} / p}-\left(\rho_{p}\left(\widetilde{u}_{n}\right)\right)^{p_{s_{\sharp}}^{*}-p}\right] \leq 0. 
\end{align}
Therefore, to obtain the strong convergence, it is sufficient to prove that
$$\left(1-c_0 \kappa\right)(\mathcal{S}(p))^{p_{s_\sharp}^{*} / p}-\left(\rho_{p}\left(\widetilde{u}_{n}\right)\right)^{p_{s_{\sharp}}^{*}-p}> 0.$$
We divide the proof into two cases: $q\in(1,p)$ and $q\in[p,p_{s_\sharp}^{*})$. Recall that  
\begin{align}\label{eq3.16}
    \mathcal{I}_{\lambda}\left(u_{n}\right)-\frac{1}{p} \langle \mathcal{I}'_{\lambda}\left(u_{n}\right), u_n \rangle = \left(\frac{1}{p}-\frac{1}{p^*_{s_\sharp}}\right) \|u_n\|_{L^{p_{s_\sharp}^{*}}(\Omega)}^{p_{s_\sharp}^{*}}-\lambda \left(\frac{1}{q}-\frac{1}{p}\right)\|u_n\|_{L^q(\Omega)}^q.
\end{align}

\noindent \textbf{Case (i) (When $1< q< p$):} From \eqref{convergences}, \eqref{ex eq 0} and \eqref{eq3.16}, we get
\begin{align}\label{eq3.22}
    c= \lim _{n \rightarrow\infty} \left(\frac{1}{p}-\frac{1}{p^*_{s_\sharp}}\right) \left(\|u\|_{L^{p_{s_\sharp}^{*}}(\Omega)}^{p_{s_\sharp}^{*}}+\|\widetilde{u}_n\|_{L^{p_{s_\sharp}^{*}}(\Omega)}^{p_{s_\sharp}^{*}}\right)-\lambda \left(\frac{1}{q}-\frac{1}{p}\right)\|u\|_{L^q(\Omega)}^q.
\end{align}
Now, using \eqref{eq3.15} in \eqref{eq3.22} and applying H\"older's and Young's inequalities, we obtain
\begin{align}\label{eq3.023}
    c \geq & \lim _{n \rightarrow\infty} \left(\frac{1}{p}-\frac{1}{p^*_{s_\sharp}}\right) \left(\|u\|_{L^{p_{s_\sharp}^{*}}(\Omega)}^{p_{s_\sharp}^{*}}+\|\widetilde{u}_n\|_{L^{p_{s_\sharp}^{*}}(\Omega)}^{p_{s_\sharp}^{*}}\right)-\lambda \left(\frac{1}{q}-\frac{1}{p}\right)|\Omega|^\frac{p^*_{s_\sharp}-q}{p^*_{s_\sharp}}\|u\|^q_{L^{p^*_{s_\sharp}}(\Omega)} \nonumber \\
    \geq &\lim _{n \rightarrow\infty} \left(\frac{1}{p}-\frac{1}{p^*_{s_\sharp}}\right) \left(\|u\|_{L^{p_{s_\sharp}^{*}}(\Omega)}^{p_{s_\sharp}^{*}}+\left(\rho_{p}(\widetilde{u}_n)\right)^{p}-\int_{[0, \bar{s}]}[\widetilde{u}_n]_{s, p}^{p} d \mu^{-}(s)\right)\nonumber\\
    &-\left(\frac{1}{p}-\frac{1}{p^*_{s_\sharp}}\right) \|u\|_{L^{p_{s_\sharp}^{*}}(\Omega)}^{p_{s_\sharp}^{*}} -|\Omega|\left(\frac{1}{p}-\frac{1}{p^*_{s_\sharp}}\right)^{-\frac{q}{p^*_{s_{\sharp}}-q}}\left[\lambda \left(\frac{1}{q}-\frac{1}{p}\right)\right]^{\frac{p^*_{s_{\sharp}}}{p^*_{s_{\sharp}}-q}} \nonumber \\
    \geq & \lim _{n \rightarrow\infty} \left(\frac{1}{p}-\frac{1}{p^*_{s_\sharp}}\right) \left(1-c_0 \kappa\right)\left(\rho_{p}\left(\widetilde{u}_{n}\right)\right)^{p} \nonumber\\
    &\quad-|\Omega|\left(\frac{1}{p}-\frac{1}{p^*_{s_\sharp}}\right)^{-\frac{q}{p^*_{s_{\sharp}}-q}}\left[\lambda \left(\frac{1}{q}-\frac{1}{p}\right)\right]^{\frac{p^*_{s_{\sharp}}}{p^*_{s_{\sharp}}-q}}.
\end{align}
Therefore, using \eqref{eq3.4} and \eqref{eq3.023}, we get
\begin{align}\label{eq3.24}
    \left(\left(1-\theta_{0}\right) \mathcal{S}(p)\right)^{N / s_{\sharp} p}> \lim _{n \rightarrow\infty}\left(1-c_0 \kappa\right)\left(\rho_{p}\left(\widetilde{u}_{n}\right)\right)^{p}.
\end{align}
Consequently, we have
\begin{align*}
& \liminf _{n \rightarrow\infty}\left(\left(1-c_0 \kappa\right)(\mathcal{S}(p))^{p_{s_{\sharp}}^{*} / p}-\left(\rho_{p}\left(\widetilde{u}_{n}\right)\right)^{p_{s_{\sharp}}^{*}-p}\right) \\
& \quad=\left(1-c_0 \kappa\right)(\mathcal{S}(p))^{N /\left(N-s_{\sharp} p\right)}-\limsup _{n \rightarrow\infty}\left(\rho_{p}\left(\widetilde{u}_{n}\right)\right)^{s_{\sharp} p^{2} /\left(N-s_{\sharp} p\right)} \\
& \quad \geq\left(1-c_0 \kappa\right)(\mathcal{S}(p))^{N /\left(N-s_{\sharp} p\right)}-\left(\frac{\left(\left(1-\theta_{0}\right) \mathcal{S}(p)\right)^{N / s_{\sharp} p}}{1-c_0 \kappa}\right)^{s_{\sharp} p /\left(N-s_{\sharp} p\right)} \\
& \quad=\left[1-c_0 \kappa-\left(\frac{\left(1-\theta_{0}\right)^{N / s_{\sharp} p}}{1-c_0 \kappa}\right)^{s_{\sharp} p /\left(N-s_{\sharp} p\right)}\right](\mathcal{S}(p))^{N /\left(N-s_{\sharp} p\right)}.
\end{align*}
Since,
\begin{align*}
   & \lim _{\kappa \rightarrow 0} \left[1-c_0 \kappa-\left(\frac{\left(1-\theta_{0}\right)^{N / s_{\sharp} p}}{1-c_0 \kappa}\right)^{s_{\sharp} p /\left(N-s_{\sharp} p\right)}\right] \\&\quad\quad\quad=1-\left(1-\theta_{0}\right)^{N /\left(N-s_{\sharp} p\right)}>0,
\end{align*}
we obtain that
\begin{equation*}
    \liminf _{n \rightarrow\infty}\left(\left(1-c_0 \kappa\right)(\mathcal{S}(p))^{p_{s_{\sharp}}^{*} / p}-\left(\rho_{p}\left(\widetilde{u}_{n}\right)\right)^{p_{s_{\sharp}^{*}}^{*}-p}\right)>0,
\end{equation*}
for sufficiently small $\kappa$ such that $c_0\kappa<1.$ Therefore, from \eqref{main ineq 2}, we deduce that 
\begin{equation*}
    \lim _{n \rightarrow\infty} \rho_{p}\left(\widetilde{u}_{n}\right)=0.
\end{equation*}
Therefore, $u_{n} \rightarrow u$ strongly in $X_{p}(\Omega)$ as $n \rightarrow\infty$.\\

\noindent \textbf{Case (ii) (When $p\leq q< p^*_s$):} Using \eqref{ex eq 0} and \eqref{eq3.16}, we get
\begin{align}\label{eq3.17}
    c\geq  \lim _{n \rightarrow\infty} \left(\frac{1}{p}-\frac{1}{p^*_{s_\sharp}}\right) \left(\|u\|_{L^{p_{s_\sharp}^{*}}(\Omega)}^{p_{s_\sharp}^{*}}+\|\widetilde{u}_n\|_{L^{p_{s_\sharp}^{*}}(\Omega)}^{p_{s_\sharp}^{*}}\right).
\end{align}
Now, from \eqref{eq3.15} and \eqref{eq3.17}, we obtain
\begin{align}\label{eq3.19}
    c\geq  \lim _{n \rightarrow\infty} \left(\frac{1}{p}-\frac{1}{p^*_{s_\sharp}}\right) \left(\left(\rho_{p}(\widetilde{u}_n)\right)^{p}-\int_{[0, \bar{s}]}[\widetilde{u}_n]_{s, p}^{p} d \mu^{-}(s)+\|u\|_{L^{p_{s_\sharp}^{*}}(\Omega)}^{p_{s_\sharp}^{*}}\right).
\end{align}
Applying \eqref{main ineq in proof} to $\widetilde{u}_{n}$ instead of $u_n$ and using \eqref{eq3.19}, we conclude that
\begin{align}\label{eq3.20}
\frac{cN}{s_{\sharp}}\geq \lim _{n \rightarrow\infty}\left(1-c_0 \kappa\right)\left(\rho_{p}\left(\widetilde{u}_{n}\right)\right)^{p} .
\end{align}
Furthermore, using \eqref{bound for level} and \eqref{eq3.20}, we get
\begin{equation}\label{eq3.21}
    \left(\left(1-\theta_{0}\right) \mathcal{S}(p)\right)^{N / s_{\sharp} p}=\frac{c^{*} N}{s_{\sharp}}>\frac{c N}{s_{\sharp}} \geq \lim _{n \rightarrow\infty}\left(1-c_0 \kappa\right)\left(\rho_{p}\left(\widetilde{u}_{n}\right)\right)^{p}.
\end{equation}
Proceeding with similar arguments as in Case (i), we conclude the strong convergence $u_n\rightarrow u$ in $X_p(\Omega)$. This completes the proof.
\end{proof}

Before proceeding to prove Theorem \ref{thm1.1}, we recall the notion of Krasnoselskii's genus and some important properties of genus \cite{R1986}. Let $X$ be a Banach space and $\Gamma$ be the collection of closed, symmetric subsets such that $0\notin A$, i.e.,
$$\Gamma=\{A\subset X: A~\text{is closed, symmetric and}~0\notin A\}.$$ 
The genus of $E\in\Gamma$ is the smallest positive integer $k$ such that there exists an odd and continuous map $h:E\rightarrow \mathbb{R}^k\setminus {0}$ and it is denoted by $\gamma (E)$. Moreover, if there is no such map, we say $\gamma(E)=\infty$. In particular, $\gamma(\emptyset)=0$.
\begin{prop}\label{prop3.3}\cite{R1986}
    Let $A$ and $B\in \Gamma$. Then we have
    \begin{itemize}
    \item[(i)] If $A\subset B$, then $\gamma(A)\leq \gamma(B).$
    \item[(ii)] If there exists an odd homeomorphism from $A$ onto $B$, then $\gamma(A)= \gamma(B).$
    \item[(iii)] $\gamma(A\cup B)\leq \gamma(A)+\gamma(B)$.
    \item [(iv)] $\gamma(\overline{A\setminus B})\geq \gamma(A)- \gamma(B)$ when $\gamma(B)<\infty$.
    \item [(v)] Genus of $k$ dimensional sphere $\mathbb{S}^{k-1}$ is $k$ by the Borsuk-Ulam theorem (see \cite{R1986,M1989}).
    \item [(vi)] If $A$ is compact then $\gamma(A)<\infty$.  Hence, there exists $\delta>0$ and a closed and symmetric neighbourhood  $N_\delta(A)=\{x\in X: \dist(x,A)\leq \delta \}$ of $A$ such that $\gamma(A)=\gamma(N_\delta(A))$.
    \item[(vii)] Let $Z$ be a subspace of $X$ with codimension $k$ and $\gamma(A)>k$, then $A \cap Z \neq \emptyset$.
    \end{itemize}
\end{prop}
 We now recall the following deformation lemma \cite[Lemma 1.3]{AR: 1973}. Let $I\in C^1(X,\mathbb{R})$. For any $c,d\in \mathbb{R}$, we denote
$$K_c=\{u\in X: I'(u)=0~\text{and}~I(u)=c\}~\text{and}~I^d=\{u\in X:I(u)\leq d\}.$$
\begin{lem}\label{lmn3.3}
Let $X$ be an infinite dimensional Banach space and let $I\in C^1(X,\mathbb{R})$ be a functional satisfying the $(\mathrm{PS})_c$ condition at level $c\in \mathbb{R}$. Let $U$ be any neighbourhood of $K_c$. Then there exist $\beta_t(x)=\beta(t,x)\in C([0,1] \times X, X )$ and constants $d_1, \epsilon>0$ with $|c|>d_1$ such that 
\begin{itemize}
    \item [(i)] $\beta_0= Id_X$, where $Id_X$ is the identity mapping in $X$. 
    \item [(ii)] $\beta_t(u)=u$ for any $u\notin I^{-1}([c-\epsilon,c+\epsilon])$ and $t\in [0,1]$.
    \item [(iii)] For any $t\in [0,1]$, $\beta_t$ is a homeomorphism.
    \item [(iv)] $I(\beta_t(u))\leq I(u)$ for all $u\in X$ and $t\in [0,1]$. 
    \item[(v)] $\beta_1(I^{c+d_1}\setminus  U)\subseteq I^{c-d_1}$.
    \item[(vi)] If $K_c=\emptyset$, then $\beta_1(I^{c+d_1})\subseteq I^{c-d_1}$.
    \item [(vii)] If $I$ is even, then $\beta_t$ is odd for any $t\in [0,1]$.
\end{itemize}
\end{lem}

\par Observe that due to the presence of the critical term, $\mathcal{I}_\lambda$ is neither coercive nor bounded from below. Therefore,  we use a truncation argument developed in \cite{AP1991} to tackle the issue posed by the critical term. Using the Sobolev constant \eqref{eq3.5}, Proposition \ref{compact and cont embedding} and Lemma \ref{lmn2.8}, we get for all $u\in X_p(\Omega)$,
\begin{align}
\mathcal{I}_{\lambda}\left(u\right)\geq & \frac{1}{p}\left(\rho_{p}\left(u\right)\right)^p-\frac{1}{p} \int_{[0, \bar{s}]}\left[u\right]_{s, p}^{p} d \mu^{-}(s)-\lambda \frac{C}{q}\left(\rho_{p}\left(u\right)\right)^q -{\frac{\mathcal{S}(p)^{-\frac{p^*_{s_\sharp}}{p}}}{p^*_{s_\sharp}}} \left(\rho_{p}\left(u\right)\right)^{p^*_{s_\sharp}} \nonumber\\
= &\frac{1}{p}\left\{\left(\rho_{p}\left(u\right)\right)^p-\int_{[0, \bar{s}]}\left[u\right]_{s, p}^{p} d \mu^{-}(s)\right\}-\lambda \frac{C}{q}\left(\rho_{p}\left(u\right)\right)^q -{\frac{\mathcal{S}(p)^{-\frac{p^*_{s_\sharp}}{p}}}{p^*_{s_\sharp}}}\left(\rho_{p}\left(u\right)\right)^{p^*_{s_\sharp}}, \nonumber \\
\geq &\frac{1}{p}(1-c_0\kappa)\left(\rho_{p}\left(u\right)\right)^p-\lambda \frac{C}{q}\left(\rho_{p}\left(u\right)\right)^q -{\frac{\mathcal{S}(p)^{-\frac{p^*_{s_\sharp}}{p}}}{p^*_{s_\sharp}}}\left(\rho_{p}\left(u\right)\right)^{p^*_{s_\sharp}}=f_\lambda(\rho_{p}\left(u\right)), \nonumber \\
\geq & \frac{1}{p}(1-c_0\kappa)\left(\eta_{p}\left(u\right)\right)^p-\lambda \frac{C}{q}\left(\eta_{p}\left(u\right)\right)^q -{\frac{C\mathcal{S}(p)^{-\frac{p^*_{s_\sharp}}{p}}}{p^*_{s_\sharp}}}\left(\eta_{p}\left(u\right)\right)^{p^*_{s_\sharp}}, \nonumber
\end{align}
where $C$ represents the differing constant values for both different and same lines and the function $f_\lambda:[0,\infty)\rightarrow \mathbb{R}$ is defined as follows:
$$f_\lambda (t)=(1-c_0\kappa)\frac{1}{p}t^p -\lambda \frac{C}{q}t^q-{\frac{\mathcal{S}(p)^{-\frac{p^*_{s_\sharp}}{p}}}{p^*_{s_\sharp}}} t^{p^*_{s_\sharp}}.$$
We fix a sufficiently small $T_1>0$ such that 
$$(1-c_0\kappa)\frac{1}{p} T_1^p-{\frac{\mathcal{S}(p)^{-\frac{p^*_{s_\sharp}}{p}}}{p^*_{s_\sharp}}} T_1^{p^*_{s_\sharp}}>0,$$
and choose $\lambda_0$ such that
\begin{equation}\label{eq3.27}
    0<\lambda_0< \frac{q}{CT_1^q}\left((1-c_0\kappa)\frac{1}{p} T_1^p-{\frac{\mathcal{S}(p)^{-\frac{p^*_{s_\sharp}}{p}}}{p^*_{s_\sharp}}} T_1^{p^*_{s_\sharp}}\right).
    \end{equation}
 Thus, we get $f_{\lambda_0}(T_1)>0$. Consider the set, $T_0=\max \{t\in (0, T_1): f_{\lambda_0}(t) \leq 0\}.$ Since, $q<p$, we have $f_{\lambda_0}(t)<0$ as $t$ approaches $0$. Therefore, using the definition of $T_0$ and the fact $f_{\lambda_0}(T_1)>0$, we conclude that $f_{\lambda_0}(T_0)=0$. Consider a nonincreasing function $\Phi\in C_c^\infty([0,\infty))$ such that $0\leq \Phi(t)\leq 1$ for all $t\in[0,\infty)$ and 
\begin{align*}
\Phi(t)= 
    \begin{cases}
        & 1 ~\text{if}~t\in [0,T_0]\\
        & 0 ~\text{if}~ t\in [T_1,\infty).
    \end{cases}
\end{align*}
 Therefore, for any $u\in X_p(\Omega)$, we define the truncated functional $\mathcal{J}_\lambda$ as follows:
\begin{align}
   \mathcal{J}_\lambda(u)=& \frac{1}{p}\left\{\left(\rho_{p}\left(u\right)\right)^p-\int_{[0, \bar{s}]}\left[u\right]_{s, p}^{p} d \mu^{-}(s)\right\}-\frac{\lambda}{q}\|u\|_{L^q(\Omega)}^q-\frac{{\Phi(\rho_p(u))}}{p_{s_\sharp}^{*}}\|u\|_{L^{p_{s_\sharp}^{*}}(\Omega)}^{p_{s_\sharp}^{*}}\\
    =&\frac{1}{p}(\eta_p(u))^p-\frac{\lambda}{q}\|u\|_{L^q(\Omega)}^q-\frac{{\Phi(\rho_p(u))}}{p_{s_\sharp}^{*}}\|u\|_{L^{p_{s_\sharp}^{*}}(\Omega)}^{p_{s_\sharp}^{*}}.
\end{align}
Clearly,  $\mathcal{J}_\lambda$ is bounded from below. Since, $\mathcal{J}_\lambda(u)\rightarrow \infty$ as $\rho_p(u) \rightarrow \infty$, conclude that $\mathcal{J}_\lambda$ is coercive.

\begin{lem}\label{lmn3.4}
    There exists $\lambda_*>0$ such that for all $\lambda \in (0, \lambda_*)$, the following statements are true.
    \begin{itemize}
        \item[(i)] If $\mathcal{J}_\lambda (u) \leq 0$ then {$\rho_p(u)<T_0$} and there exists a neighbourhood $U$ of $u$ such that $\mathcal{I}_\lambda(v)=\mathcal{J}_\lambda(v)$ for all $v\in U$.
        \item [(ii)] $\mathcal{J}_\lambda$ is satisfy $(\mathrm{PS})_c$ condition for any $c<0$.
    \end{itemize}
    \end{lem}
\begin{proof}
    Recall, $\lambda_0$ as in \eqref{eq3.27}. We define
    \begin{align}\label{eq3.28}
        \tilde{\lambda}=\frac{q(1-c_0\kappa)}{pC}T_1^{p-q}
    \end{align}
    and choose $\lambda'$ sufficiently small such that 
    \begin{align}\label{eq3.29}
        \frac{s_{\sharp}}{N}\left(\left(1-\theta_{0}\right) \mathcal{S}(p)\right)^{N / ps_{\sharp}}-|\Omega|\left(\frac{s_\sharp}{N}\right)^{-\frac{q}{p^*_{s_{\sharp}}-q}}\left[\lambda' \left(\frac{1}{q}-\frac{1}{p}\right)\right]^{\frac{p^*_{s_{\sharp}}}{p^*_{s_{\sharp}}-q}}>0.
         \end{align}
    Fix $\lambda \in (0, \lambda_*)$, where $\lambda_* \leq \min\{\lambda_0, \tilde{\lambda}, \lambda' \}$. We divide the proof into two parts.\\

    \noindent   \textbf{Part I:} Let $\mathcal{J}_\lambda(u)\leq 0$. For $\rho_p(u) \geq T_1$, we use the Sobolev embedding \eqref{embedding} to obtain
    $$\mathcal{J}_\lambda(u)\geq (1-c_0\kappa)\frac{1}{p}(\rho_p(u))^p-\frac{\lambda C}{q}(\rho_p(u))^q=g_\lambda(\rho_p(u)),$$
    where the continuous function $g_\lambda:[0,\infty)\rightarrow \mathbb{R}$ is defined as follows
    $${g_\lambda(t)=(1-c_0\kappa)\frac{1}{p}t^p-\frac{\lambda C}{q}t^q.}$$
    Observe that $g_\lambda$ has just two roots $t_0=0$ and $t_1=(\frac{\lambda p C}{q(1-c_0\kappa)})^\frac{1}{p-q}$. Moreover, $g_\lambda(t) \to \infty$ as $t \to \infty$, since $p > q$. Now, $\lambda<\tilde{\lambda}$ implies $t_1<T_1$. Thus, we get $0< \mathcal{J}_\lambda(u)\leq 0,$ which constitutes a contradiction. 
    \par When $\rho_p(u) < T_1$, using $0\leq \Phi(t)\leq 1 $ and $\lambda<\lambda_0$, we obtain
    $$0 \geq \mathcal{J}_\lambda(u)\geq f_\lambda(\rho_p(u))> f_{\lambda_0}(\rho_p(u)).$$
    Therefore, using the definition of $T_0$, we conclude that $\rho_p(u)<T_0$. In addition, using the continuity of the functional $\mathcal{J}_\lambda$, there exists a sufficiently small neighbourhood $U\subset B_{T_0}(0)$ of $u$ such that $\mathcal{J}_\lambda(v)<0$. Therefore, for any $v \in U \subset B_{T_0}(0)$, it follows that $\mathcal{J}_\lambda(v) = \mathcal{I}_\lambda(v)$.\\ 
    
    \noindent  \textbf{Part II:} Let $c<0$ and $(u_n)$ be a $(\mathrm{PS})_c$ sequence for $\mathcal{J}_\lambda$ at the level $c$. From \textbf{Part I}, we get $\mathcal{J}_\lambda(u_n) = \mathcal{I}_\lambda(u_n)$ for sufficiently large $n\in \mathbb{N}$. Since, $\mathcal{J}_\lambda$ is coercive, the sequence $(u_n)$ is bounded in $X_p(\Omega)$. Therefore, for $\lambda<\lambda'$ as stated in $\eqref{eq3.29}$ and from Lemma \ref{PS condition lem}, we conclude that $\mathcal{J}_\lambda$ satisfies the $(\mathrm{PS})_c$ condition. This completes the proof.
\end{proof}

\begin{lem}\label{lmn3.5}
    Let $k\in \mathbb{N}$ and $\lambda>0$. Then there exists $\epsilon=\epsilon(k,\lambda)>0$ depending on $\lambda$ and $k$, such that $\gamma(\mathcal{J}_\lambda^{-\epsilon})\geq k$, where $\mathcal{J}_\lambda^{-\epsilon}:=\{u\in X_p(\Omega): \mathcal{J}_\lambda(u)\leq -\epsilon\}$. 
\end{lem}
\begin{proof}
    Recall that every norm is equivalent in finite dimensional normed linear space. Let $k\in \mathbb{N}$ and $\lambda>0$. Let $\mathcal{V}_k$  be a $k$-dimensional subspace of $X_p(\Omega)$. Therefore, $\rho_p(.)$ and $\|.\|_{L^q(\Omega)}$ are equivalent norms in $\mathcal{V}_k$ for every fixed $k$. Thus, there exists a constant $C=C(k)>0$ such that 
    \begin{equation}\label{eq3.30}
    C\rho_p(u)\leq \|u\|_{L^q(\Omega)},~ \forall~ u\in \mathcal{V}_k \text{ and } 1\leq q<\infty.
    \end{equation}
    For any $u\in \mathcal{V}_k$ with $\rho_p(u) \leq T_0$ and using \eqref{eq3.30} and Lemma \ref{lmn2.8}, we get
    \begin{align}\label{eq3.31}
        \mathcal{I}_\lambda(u)=\mathcal{J}_\lambda(u) \leq & \frac{1}{p}\left\{\left(\rho_{p}\left(u\right)\right)^p-\int_{[0, \bar{s}]}\left[u\right]_{s, p}^{p} d \mu^{-}(s)\right\}-\frac{\lambda}{q}\|u\|_{L^q(\Omega)}^q \nonumber \\
         \leq & \frac{1}{p} \left\{\left(\rho_{p}\left(u\right)\right)^p-\int_{[0, \bar{s}]}\left[u\right]_{s, p}^{p} d \mu^{-}(s)\right\} -\frac{\lambda C}{q} (\rho_p(u))^q\nonumber \\
        \leq & \frac{1}{p} (\eta_p(u))^p -\frac{\lambda C}{q} (\rho_p(u))^q \leq \frac{1}{p} (\rho_p(u))^p -\frac{\lambda C}{q} (\rho_p(u))^q .
     \end{align}
    Let $S_a:=\{u\in \mathcal{V}_k: \rho_p(u)=a \},$ where the constant $a>0$ is such that
    \begin{equation}\label{eq3.32}
        0<a< \min \left\{ T_0,\left( \frac{\lambda C p}{q}\right)^\frac{1}{p-q}\right\}.
    \end{equation}
    Observe that $S_a$ is odd homeomorphic to the $(k-1)$-dimensional sphere $\mathbb{S}^{k-1}$. Furthermore, employing \eqref{eq3.31} and \eqref{eq3.32} for each $u\in S_a$, we obtain
    \begin{equation}
        \mathcal{J}_\lambda(u) \leq a^q \left( \frac{1}{p} a^{p-q} -\frac{\lambda C}{q}\right)<0.
    \end{equation}
    Therefore, for any $u\in S_a$, there exists a constant $\epsilon(\lambda, k)>0$ such that $\mathcal{J}_\lambda(u)<-\epsilon$. Also, we have $S_a \subset \mathcal{J}_\lambda^{-\epsilon} $. Thus, using $(i),(ii)$ and $(iv)$ of Proposition \ref{prop3.3}, we get
    $$\gamma(\mathcal{J}_\lambda^{-\epsilon})\geq \gamma (S_a)=k.$$
    This completes the proof.
\end{proof}

For any $k\in \mathbb{N}$, we define the min-max values $c_k$ as,
\begin{equation}\label{eq3.34}
    c_k=\inf_{A\in \Sigma_k} \sup_{u \in A} \mathcal{J}_\lambda (u),
\end{equation}
where 
$\Sigma_k=\{A\subset X_p(\Omega): A ~\text{symmetric and closed such that} ~0\notin A ~\text{with}~ \gamma(A)\geq k \}.$
It is important to note that $c_k \leq c_{k+1}$ for all $k \in \mathbb{N}$.

\begin{lem}\label{lmn3.6}
    For each $\lambda>0$ and $k \in \mathbb{N}$, we have $c_k<0$.
\end{lem}

\begin{proof}
    Fix $\lambda>0$ and $k \in \mathbb{N}$. From Lemma \ref{lmn3.4}, there exists $\epsilon>0$ such that $\gamma(\mathcal{J}_\lambda^{-\epsilon})\geq k$. Since, 
 $\mathcal{J}_\lambda^{-\epsilon}$ is closed and symmetric, we have $\mathcal{J}_\lambda^{-\epsilon} \in \Sigma_k$. Moreover, $0\notin \mathcal{J}_\lambda^{-\epsilon}$ since $\mathcal{J}_\lambda(0)=0$ and $\sup_{u\in \mathcal{J}_\lambda^{-\epsilon}} \mathcal{J}_\lambda (u) \leq -\epsilon$. Consequently, using the fact that $\mathcal{J}_\lambda$ is bounded from below, we derive 
 $$-\infty < c_k = \inf_{A \in \Sigma_k} \sup_{u \in A} \mathcal{J}_\lambda (u) \leq \sup_{u \in \mathcal{J}_\lambda^{-\epsilon}} \mathcal{J}_\lambda (u) \leq -\epsilon<0.$$
 This completes the proof.
\end{proof}

\begin{lem}\label{lmn3.7}
    Let $\lambda \in(0, \lambda_*)$, where $\lambda_*$ is as in Lemma \ref{lmn3.4} and $k\in \mathbb{N}$. If $c=c_k=c_{k+1}=c_{k+2}=...=c_{k+m}$ for some $m\in \mathbb{N}$, then $\gamma(K_c)\geq m+1,$ where $K_c=\{u\in X_p(\Omega): \mathcal{J}_\lambda(u)=c~\text{and}~\mathcal{J'}_\lambda(u)=0\}$. 
\end{lem}
\begin{proof}
   Let $\lambda \in (0, \lambda_*)$ and $k \in \mathbb{N}$. Applying Lemma \ref{lmn3.6}, we get $c=c_k=c_{k+1}=c_{k+2}=...=c_{k+m}<0$. Thanks to Lemma \ref{lmn3.4}, we conclude that $K_c$ is compact. Now, suppose $\gamma(K_c)\leq m$. Therefore, from Proposition \ref{prop3.3}-(vi), there exist $\delta>0$ and a neighbourhood $N_\delta(K_c)$ of $K_c$ such that $\gamma(N_\delta(K_c))=\gamma(K_c)\leq m$, where
   $$N_\delta(K_c)=\{w \in X_p(\Omega):\dist(w,K_c)\leq \delta\}.$$
   From Lemma \ref{lmn3.3}, there exist $\epsilon\in (0,-c)$ and odd homeomorphism $\beta:X_p(\Omega)\rightarrow X_p(\Omega)$ such that
    \begin{equation}\label{eq3.340}
        \beta(\mathcal{J}_\lambda^{c+\epsilon}\setminus N_\delta(K_c))\subset \mathcal{J}_\lambda^{c-\epsilon}.
        \end{equation}
      Now using the definition of infimum and $c_{k+m}=c$ as in \eqref{eq3.34}, there exists $A\in \Sigma_{k+m}$ such that 
      $$\sup_{u\in A}\mathcal{J}_\lambda(u)<c+\epsilon,$$
      implying that $A\subset \mathcal{J}_\lambda^{c+\epsilon}$. Therefore, we get 
      \begin{align}\label{eq3.35}
          \beta(A \setminus N_\delta(K_c))\leq \beta(\mathcal{J}_\lambda^{c+\epsilon}\setminus N_\delta(K_c))\subset \mathcal{J}_\lambda^{c-\epsilon},
      \end{align}
      which implies $\mathcal{J}_\lambda (\beta(A \setminus N_\delta(K_c))) \leq c-\epsilon$. On the other hand, using $(i)$ and $(iii)$ of Proposition \ref{prop3.3}, we obtain 
      $$\gamma(\beta(\overline{A \setminus N_\delta(K_c)}))\geq \gamma(\overline{A \setminus N_\delta(K_c)})\geq \gamma(A)-\gamma(N_\delta(K_c))\geq k.$$
      Therefore, we have $\beta (\overline{A \setminus N_\delta(K_c)})\in \Sigma_k$. From the definition of $c_k$, we get 
    $$\sup_{u\in \beta (\overline{A \setminus N_\delta(K_c)})} \mathcal{J}_\lambda(u) \geq c_k=c, $$
      which is a contradiction to \eqref{eq3.35}. Therefore, $\gamma(K_c)\geq m+1$. This completes the proof. 
\end{proof}
We have now obtained the essential results to prove our main theorem.

\begin{center}
    Proof of Theorem \ref{thm1.1}
\end{center}

Let $\lambda \in (0, \lambda_*)$, where $\lambda_*$ is defined in Lemma \ref{lmn3.4}. From Lemma \ref{lmn3.6}, it follows that $c_k<0$.  Furthermore, Lemma \ref{lmn3.4}-(ii) guarantees that the functional $\mathcal{J}_\lambda$ satisfies the $(\mathrm{PS})_{c_k}$ condition. Therefore, we conclude that $c_k$ is a critical value of $\mathcal{J}_\lambda$ for every $k\in \mathbb{N}$, (see Rabinowitz \cite{R1986}). Now, there are two cases for the critical values $c_k$ that occur.\\

\noindent \textbf{Case I:} When $-\infty<c_1<c_2<\ldots<c_k<c_{k+1}<\ldots$. In this case, $\mathcal{J}_\lambda$ has infinitely many critical values.\\

\noindent \textbf{Case II:-} When there exist $k,m\in \mathbb{N}$ such that $c_k=c_{k+1}=c_{k+2}=....c_{k+m}=c$. Thanks to Lemma \ref{lmn3.7}, we get $\gamma(K_c)\geq m+1\geq 2$. Therefore, using \cite[Remark 7.3]{R1986} and Lemma \ref{lmn3.4}-(ii), we deduce that $K_c$ possesses infinitely many critical points of the functional $\mathcal{J}_\lambda$.

\par Therefore, using Lemma \ref{lmn3.4}-(i), we obtain infinitely many negative critical values for $\mathcal{J}_\lambda=\mathcal{I}_\lambda$. Hence, the problem \eqref{problem 2} has infinitely many weak solutions in the sense of Definition \ref{weaksol MP}. This completes the proof.

\section{Important examples and applications} \label{secexpa}

In this section, we build upon the main results established earlier to examine several noteworthy cases, each arising from a particular choice of the measure $\mu.$ The notation $\mathbb{X}(\Omega)$ is used to represent the appropriate Sobolev space for each example. Since the underlying operator varies from case to case, this notation should be interpreted accordingly and specified precisely in terms of $X_p(\Omega).$\\

\textbf{The Laplace and $p$-Laplacian:-}\label{subsec2.1}
At first, we will provide the results of the classical Laplacian $(-\Delta)$ and the classical $p$-Laplacian $-\Delta_{p}$.  Indeed, with a particular choice of the measure $\mu^+$, Theorem \ref{thm1.1} re-establishes the main results from \cite[Theorem 4.5]{AP1991}.
    \begin{proof}
        We choose $\mu=\delta_{1},$ the Dirac measure (unit mass) centred at the point $1.$  Consequently, it fulfils all conditions \eqref{measure 1}-\eqref{measure 3} with $\bar{s}=1$ and $\kappa=0.$  Therefore, we select $s_\sharp:=1$ as indicated in \eqref{measure 4}.  The proof immediately derives from Theorem \ref{thm1.1}.
    \end{proof}

\textbf{Fractional Laplacian and Fractional $p$-Laplacian:-} \label{subsec2.2}
The next results of the fractional Laplacian $(-\Delta)^s$ and fractional $p$-Laplacian $(-\Delta)^s_p$. Indeed, with an individual choice of the measure $\mu^+$, Theorem \ref{thm1.1} recovers the main results of \cite[Theorem 1.4]{YZ2024}.

\begin{proof}
   For $s \in (0, 1),$ we take the measure $\mu:=\delta_s,$ where $\delta_s$ denotes the Dirac measure centred at $s.$  It is thus straightforward to verify that $\mu$ fulfils all the conditions \eqref{measure 1}-\eqref{measure 3} with $\bar{s}=s$ and $\kappa=0.$  In this context, we choice $s_\sharp:=s$ and the proof follows directly from Theorem \ref{thm1.1}.
\end{proof}

\textbf{Mixed local and nonlocal operators and Mixed local and nonlocal $p$- Laplacian operators:-} \label{subsec2.3}
Let us now give results related to the mixed local and nonlocal operator $( -\Delta+(-\Delta)^s)$ and mixed local and nonlocal $p$-Laplace operator $( -\Delta_{p} +(-\Delta_{p})^s)$.  Furthermore, Theorem \ref{thm1.1} is immediately derived from \cite[Theorem 1.1]{DFV24}.

\begin{proof}
   In this situation, we choose $\mu:=\delta_1+\delta_s,$ where $\delta_1$ and $\delta_s$ denote the Dirac measures centred at $1$ and $s \in (0, 1),$ respectively.  Consequently, $\mu$ satisfying all of the conditions \eqref{measure 1}-\eqref{measure 3} with $\bar{s}:=1$ and $\kappa:=0.$  We now take $s_\sharp:=1$, and the proof is derived from Theorem \ref{thm1.1}.
\end{proof}
\begin{rem}
    Another generalized mixed local and nonlocal operator is $ -\Delta_{p} +\epsilon(-\Delta_{p})^s$ where $\epsilon\in (0,1]$.
\end{rem}
\begin{proof}
   In this context, we select $\mu := \delta_1 + \epsilon \delta_s$, where $\delta_1$ and $\delta_s$ represent the Dirac measures centered at $1$ and $s \in (0, 1)$, respectively.   Thus, $\mu$ fulfills all conditions \eqref{measure 1}-\eqref{measure 3} with $\bar{s}:=1$ and $\kappa:=0.$   We now set $s_\sharp:=1$, and the proof is founded on Theorem \ref{thm1.1}.
\end{proof}

\textbf{Nonlocal operators associated with a convergent series of Dirac measures} \label{subsec2.4}
We now show that, in our context, we can also choose $\mu$ as a convergent series of Dirac measures.
\begin{cor}  Let $ \Omega$ be a bounded subset of $\mathbb{R}^N$ and $1 \geq s_0>s_1>s_2> \ldots \geq 0.$ Assume that the operator 
    $$\sum_{k=0}^{+\infty} a_k (-\Delta_p)^{s_k}\quad \text{with} \quad \sum_{k=0}^{+\infty} a_k \in (0, +\infty),$$
    where $a_0>0$ and $a_k \geq 0$ for all $k \geq 1$ with $p_{s_0}^*:= \frac{pN}{N-ps_{0}}$ be the fractional critical Sobolev exponent and $1<q <p<\frac{N}{s_0}.$  Then, there exist $\lambda_*$ and $\kappa_*$ such that for all $\lambda \in (0, \lambda_*)$ and for all $\kappa \in [0, \kappa_*]$, the problem \begin{align}\label{problem52}
\sum_{k=0}^{+\infty} a_k (-\Delta_p)^{s_k}  u &= \lambda |u|^{q-2}u+|u|^{p^*_{s_0}-2}u ~\text{in}~\Omega,  \nonumber\\
 u & =0 ~\text{in}~ \mathbb{R}^N \setminus \Omega,
 \end{align} possesses infinitely many nontrivial solutions in $\mathbb{X}(\Omega)$ with negative energy.
\end{cor}

\begin{proof}
 For applying Theorem \ref{thm1.1}, we define 
 $$\mu:= \sum_{k=0}^{+\infty} a_k \delta_{s_k},$$
 where $\delta_{s_k}$ represents the Dirac measures at $s_k.$   Subsequently, we observe that $\mu$ fulfills all the conditions \eqref{measure 1}-\eqref{measure 3} by choosing 
 $\bar{s}:=s_0$, $\kappa:=0$ and $s_\sharp := s_0$.
 \end{proof}

\begin{cor}  Let $ \Omega$ be a bounded subset of $\mathbb{R}^N$ and $1 \geq s_0>s_1>s_2> \ldots \geq 0.$ Assume that the operator 
    $$\sum_{k=0}^{+\infty} a_k (-\Delta_p)^{s_k}\quad \text{with} \quad \sum_{k=0}^{+\infty} a_k \in (0, +\infty).$$
    There exist $\kappa \geq 0$ and $\bar{k} \in \mathbb{N}\cup \{0\}$ such that 
    \begin{equation} \label{condis_k}
        a_k>0 \quad \forall \,\, k \in \{0, 1, \ldots, \bar{k}\}, \quad \text{and}\,\,\, \sum_{k=\bar{k}+1}^{+\infty} a_k \leq \kappa \sum_{k=0}^{\bar{k}} a_k.
    \end{equation}
    Consider that $p_{s_0}^*:= \frac{pN}{N-ps_{0}}$ be the fractional critical Sobolev exponent and  $1<q <p<\frac{N}{s_0}.$  Then, there exist $\lambda_*$ and $\kappa_*$ such that for all $\lambda \in (0, \lambda_*)$ and for all $\kappa \in [0, \kappa_*]$, the problem \begin{align}\label{problem53}
\sum_{k=0}^{+\infty} a_k (-\Delta_p)^{s_k}  u &= \lambda |u|^{q-2}u+|u|^{p^*_{s_0}-2}u ~\text{in}~\Omega,  \nonumber\\
 u & =0 ~\text{in}~ \mathbb{R}^N \setminus \Omega,
 \end{align} possesses infinitely many nontrivial solutions in $\mathbb{X}(\Omega)$ with negative energy.
\end{cor}
\begin{proof} 
We define the measure $\mu$ as follows: 
  $$\mu := \sum_{k=0}^{+\infty} a_k \delta_{s_k},$$ 
  where $\delta_{s_k}$ represents the Dirac measures concentrated at $s_k.$

 By taking $\bar{s}=s_{\bar{k}}$ and $s_\sharp:=s_0$, and applying condition \eqref{condis_k} for $\kappa$, we can easily verify that $\mu$ fulfils all the conditions \eqref{measure 1}-\eqref{measure 3}.  Consequently, we obtained the result as a consequence of Theorem \ref{thm1.1}.
\end{proof}

\textbf{Mixed local and nonlocal operator with wrong sign} \label{subsec2.5}
An attractive scenario arises when the measure $\mu$ changes sign.  This means, for example, that the operator could use a minor term with a ``wrong" sign. To our knowledge, no existing literature addresses a result of this nature, even for the situation $p=2.$

\begin{cor}
Let $ \Omega$ be a bounded subset of $\mathbb{R}^N$ and let $s \in [0, 1)$. Suppose that $N>p >1$, $q \in(1, p)$ and critical Sobolev exponent is $p^*=\frac{p N}{N-p}$. Assume that the operator  $ -\Delta_{p} u-\alpha (-\Delta_{p})^s u$, where $\alpha$ is sufficiently small. Then, there exist $\lambda_*$ and $\kappa_*$ such that for all $\lambda \in (0, \lambda_*)$ and for all $\kappa \in [0, \kappa_*]$, the problem \begin{align}\label{problem54}
 -\Delta_{p}u -\alpha (-\Delta_{p})^s u  &= \lambda |u|^{q-2}u+|u|^{p^*-2}u ~\text{in}~\Omega,  \nonumber\\
 u & =0 ~\text{in}~ \mathbb{R}^N \setminus \Omega,
 \end{align} possesses infinitely many nontrivial solutions in $\mathbb{X}(\Omega)$ with negative energy.
\end{cor}
\begin{proof}
   For the proof, we define $\mu:=\delta_1-\alpha \delta_s,$ where $\delta_1$ and $\delta_s$ are Dirac measures centred at $1$ and $s \in [0, 1),$ respectively.  We now choose $\bar{s}:=1$ and $s_\sharp:=1$, making sure both \eqref{measure 1} and \eqref{measure 2} are satisfied.  Furthermore, see that $$\mu^-([0, \bar{s}]) \leq \max \{0, \alpha\} = \max\{0, \alpha\} \mu^+([\bar{s}, 1]),$$ suggesting that condition \eqref{measure 3} is satisfied by choosing $\kappa:= \max\{0, \alpha\}. $  Consequently, the proof is a direct application of Theorem \ref{thm1.1}. 
\end{proof}

\begin{cor}
    Let $ \Omega$ be a bounded subset of $\mathbb{R}^N$ and let  $N>p >1$.  Let $\alpha \geq 0$ and $1>s_1>s_2>0.$ Assume that the operator $-\Delta_{p} u+ (-\Delta_{p})^{s_1} u-\alpha (-\Delta_{p})^{s_2} u$ and $1< q<p<p^*=\frac{p N}{N-p}.$ Then, there exist $\lambda_*$ and $\kappa_*$ such that for all $\lambda \in (0, \lambda_*)$ and for all $\kappa \in [0, \kappa_*]$, the problem \begin{align}\label{problem55}
 -\Delta_{p} u+ (-\Delta_{p})^{s_1} u-\alpha (-\Delta_{p})^{s_2} u  &= \lambda |u|^{q-2}u+|u|^{p^*-2}u ~\text{in}~\Omega,  \nonumber\\
 u & =0 ~\text{in}~ \mathbb{R}^N \setminus \Omega,
 \end{align} possesses infinitely many nontrivial solutions in $\mathbb{X}(\Omega)$ with negative energy.

\end{cor}
\begin{proof} Define $\mu:= \delta_1+\delta_{s_1}-\alpha \delta_{s_2}$ for $1>s_1>s_2>0.$ Subsequently, both the conditions \eqref{measure 1} and \eqref{measure 2} are satisfied 
 for $\bar{s}:=s_1$ and $s_\sharp:=1.$ Now, we will verify conditions \eqref{measure 3} hold true for sufficiently small $\alpha.$ Indeed, we observe that 
$$\mu^-([0, s_1)) = \alpha = \frac{\alpha}{2} \mu^+([s_1, 1]),$$
which implies that condition \eqref{measure 3} is hold for any $\kappa \geq \frac{\alpha}{2}.$  Consequently, the proof of the results is derived from Theorem \ref{thm1.1}.
\end{proof}

\textbf{Nonlocal operator driven by continuous superposition of the fractional $p$-Laplacian} \label{subsec2.6}
We note another interesting result arising from the continuous superposition of fractional operators of the $p$-Laplacian type.  To the best of our knowledge, this result is also new for $p=2$.

\begin{cor}
 Let $ \Omega$ be a bounded subset of $\mathbb{R}^N.$    Let $s_\sharp \in (0, 1), \kappa \geq 0,$  and let $f \not\equiv 0 $ be a measurable function such that 
    \begin{align} \label{condf}
        &f \geq 0 \quad \text{in} \quad (s_\sharp, 1), \nonumber \\
        &\int_{s_\sharp}^1 f(s)\, ds >0, \\
        \text{and}\quad &\int_0^{s_\sharp} \max\{0, -f(s)\}\, ds \leq \kappa \int_{s_\sharp}^1 f(s)\, ds. \nonumber
        \end{align}
   Consider the operator defined as $\bigintss_0^1 f(s) (-\Delta_p)^s u\, ds$, where $\frac{N}{s_{\sharp}}>p > 1$, $q \in(1, p)$, and $p_{s_\sharp}^*=\frac{p N}{N-{s_\sharp}p}$ represents the fractional critical Sobolev exponent. Then, there exist $\lambda_*$ and $\kappa_*$ such that for all $\lambda \in (0, \lambda_*)$ and for all $\kappa \in [0, \kappa_*]$, the problem 
   \begin{align}\label{problem56}
 \int_0^1 f(s) (-\Delta_p)^s u\, ds  &= \lambda |u|^{q-2}u+|u|^{p^*_{s_\sharp}-2}u ~\text{in}~\Omega,  \nonumber\\
 u & =0 ~\text{in}~ \mathbb{R}^N \setminus \Omega,
 \end{align} possesses infinitely many nontrivial solutions in $\mathbb{X}(\Omega)$ with negative energy.
    
\end{cor}
\begin{proof}
    In this instance, we define $d\mu(s):=f(s) ds$, where $f$ is as defined in the definition of the result.  Consequently, the operator $A_{p, \mu}$ is expressed as $$ \int_0^1 f(s) (-\Delta_p)^s u\, ds.$$
     Due to the conditions stated in \eqref{condf}, all the conditions \eqref{measure 1}-\eqref{measure 3} are fulfilled by choosing $\bar{s}:=s_\sharp,$ which also works as the crucial fractional Sobolev exponent.   Consequently, the proof of this result is derived from Theorem \ref{thm1.1}.
\end{proof}


\section*{Conflict of interest statement}
On behalf of all authors, the corresponding author states that there is no conflict of interest.

\section*{Data availability statement}
Data sharing is not applicable to this article as no datasets were generated or analysed during the current study.

\section*{Acknowledgement}
The authors thank the reviewers for their careful reading and constructive suggestions, which improved the quality of the paper. SB would like to thank the Council of Scientific and Industrial Research (CSIR), India for financial assistance to carry out this research work [grant no. 09/0874(17164)/ 2023-EMR-I]. SG gratefully acknowledges the financial support for this research work under ARG-MATRICS, grant No: ANRF/ARGM/2025/ 001570/MTR, Anusandhan National Research Foundation (ANRF), Government of India. SG would also like to thank the Ghent Analysis \& PDE Centre, Ghent University, Belgium, for the support during his research visit.  VK is supported by the FWO Odysseus 1 grant G.0H94.18N: Analysis and Partial Differential Equations and the Methusalem program of the Ghent University Special Research Fund (BOF) (Grant number 01M01021). VK is also supported by FWO Senior Research Grant G011522N.

\end{document}